\documentclass[reqno,twoside]{amsart}
\usepackage{amsfonts,amssymb,amsmath}
\usepackage{mathrsfs}
\usepackage{enumerate}
\allowdisplaybreaks

 \newtheorem{thm}{Theorem}[section]
 \newtheorem{cor}[thm]{Corollary}
 \newtheorem{lem}[thm]{Lemma}
  
 \newtheorem{prop}[thm]{Proposition}
 \theoremstyle{definition}
 \newtheorem{defn}[thm]{Definition}
 \theoremstyle{remark}
 \newtheorem{rem}[thm]{Remark}
 \newtheorem*{ex}{Example}
 \numberwithin{equation}{section}

\newcommand{\R}{{\mathbb R}}

\newcommand{\N}{{\mathbb N}}

\newcommand{\E}{{\mathbb E}}
\newcommand{\p}{{\mathbb P}}

\newcommand{\B}{{\mathcal B}}

\newcommand{\one}{1\!\!\!\;\mathrm{l}}
\newcommand{\eps}{\varepsilon}
\def\oo{\mathaccent23}
\def\ds{\displaystyle}
\def\eps{\varepsilon}

\numberwithin{equation}{section}
\begin{document}

\title[Dirichlet semigroups in Hilbert spaces]
 {On the Dirichlet semigroup for Ornstein -- Uhlenbeck operators in  subsets  of   Hilbert spaces}


\author[G. Da Prato]{Giuseppe Da Prato}
\address{Scuola Normale Superiore\\
Piazza dei Cavalieri, 7\\ 
56126 Pisa, Italy}
\email{g.daprato@sns.it}

\author[A. Lunardi]{Alessandra Lunardi}
\address{
Dipartimento di Matematica\\
Universit\`a di Parma\\
Viale G.P. Usberti, 53/A\\
43124 Parma, Italy}
\email{alessandra.lunardi@unipr.it}



\begin{abstract}
We consider a family  of self-adjoint Ornstein--Uhlenbeck operators ${\mathcal L}_{\alpha} $ in an infinite dimensional Hilbert space $H$ having   the same gaussian invariant measure $\mu$ for all $'\alpha \in [0,1]$.
We study the   Dirichlet problem for the equation $\lambda \varphi -  {\mathcal L}_{\alpha} \varphi = f$  in a closed set $K$, with $f\in L^2(K, \mu)$.  We first 
prove that the variational solution, trivially provided by the Lax---Milgram theorem, can be represented, as  expected, by means  of  the transition semigroup stopped to $K$.  
Then we address two problems: 1)  the regularity of the solution $\varphi$ (which is by definition in a Sobolev space $W^{1,2}_{\alpha}(K,\mu)$) of the Dirichlet problem; 2) the meaning of the Dirichlet boundary condition. Concerning regularity, we are able to prove interior  $W^{2,2}_{\alpha}$ regularity results;  concerning the boundary condition we consider both irregular and regular boundaries. In the first case we content to have a solution whose null extension outside $K$ belongs to $W^{1,2}_{\alpha}(H,\mu)$. In the second case we
 exploit the Malliavin's theory of surface integrals which is recalled in the Appendix of the paper, 
then we are able to give a meaning to the trace of $\varphi$ at $\partial K$ and to  show that it vanishes, as it is natural.
\end{abstract}

\maketitle
\section{Introduction and setting of the problem}

In this paper we present some results on second order elliptic and parabolic equations with  Dirichlet boundary conditions in a closed set  of a separable real Hilbert space 
$H$ (norm $|\cdot|$, inner product  $\langle \cdot,\cdot \rangle$).  

A  motivation  for the study of  Dirichlet  problems in proper subsets  of  $H$ is 
to provide a natural development of the potential theory in infinite dimensions started  in \cite{Gross}. 
Only a few results seem to be available in this field,   see e.g. \cite{DPZ3} and the references therein. 

  The finite dimensional theory in spaces of continuous functions is hardly extendable to the infinite dimensional setting. While 
  in finite dimensions smooth boundaries consist only of regular points in the sense of Wiener, 
  in infinite dimensions this is not true: for instance, certain hyperplanes and   the boundary of the unit ball  contain  dense subsets of irregular points for suitable Ornstein-Uhlenbeck operators (\cite{DPGZ}). This   leads to the lack of regularity results up to the boundary.  

Here we avoid a part of such  difficulties working in suitable $L^2$ spaces. 
 
To begin with, we consider  a class of Ornstein--Uhlenbeck operators  of the type
\begin{equation}
\label{OU}
{\mathcal L}_{\alpha} \varphi(x)=\frac{1}{2}\;\mbox{\rm Tr}\;[Q^{1-\alpha}D^2\varphi(x)] -\frac{1}{2}\langle x,Q^{-\alpha}D \varphi(x)  \rangle, 
\end{equation}
where $Q\in {\mathcal L}(H)$ is  a symmetric positive operator with finite trace, and $0\leq \alpha \leq 1$.

The most popular among such operators are ${\mathcal L}_{0} $ and ${\mathcal L}_{1} $:
$${\mathcal L}_{0}  \varphi(x)=\frac12\;\mbox{\rm Tr}\;[Q D^2\varphi(x)] -\frac{1}{2}\langle x, D \varphi(x)  \rangle, $$
is the operator that arises in the Malliavin calculus, 
while
$${\mathcal L}_{1} \varphi(x)=\frac12\;\mbox{\rm Tr}\;[ D^2\varphi(x)] -\frac{1}{2}\langle x,AD \varphi(x)  \rangle, $$
(with $A= Q^{-1}$) is the generator of the Ornstein-Uhlenbeck semigroup with the best smoothing properties. See e.g. \cite{DPZ3}.

The operators ${\mathcal L}_{\alpha} $ exhibit an important common feature:  the associated Ornstein-Uhlenbeck semigroups $T_{\alpha}(t) $ in $C_b(H)$ have the same invariant measure $\mu={\mathcal N}_{Q}$, the Gaussian measure  of  mean  $0$ and covariance $Q $. 
In this paper  we shall consider realizations of the operators ${\mathcal L}_{\alpha} $ in the space $L^2(K , \mu)$, where $K$ is a closed set   in $H$ with non empty interior part $\oo{K}$.

A unique weak solution to the Dirichlet problem
\begin{equation}
\label{e1.6a}
\left\{\begin{array}{l}
\lambda \varphi(x)-{\mathcal L}_{\alpha} \varphi(x)=f(x), \quad\mbox{\rm in}\;  K,
\\
\\
\varphi(x)=0,\quad\mbox{\rm on}\; \partial K 
\end{array}\right.
\end{equation}
with $\lambda >0$ and $f\in L^2(K,\mu)$ is easily obtained via the Lax-Milgram Theorem, applied in a Hilbert space $\oo{W}^{1,2}_{\alpha} (K ,\mu)$ ``naturally" associated to ${\mathcal L}_{\alpha} $ (see next section). 
This allows to define a dissipative self-adjoint operator $M_{\alpha}$ in $L^2(K,\mu)$ such that $\varphi = R(\lambda, M_{\alpha})f$. As all dissipative self-adjoint operators in Hilbert spaces, $M_{\alpha}$ is the infinitesimal generator of an analytic contraction semigroup.

We give an explicit expression of  the semigroup generated by $M_{\alpha}$. Precisely, we identify it with the natural extension  to $L^2(K,\mu)$ of the so-called {\em stopped semigroup}  $T^{K}_{\alpha}(t)$.  In analogy with  the finite dimensional case (e.g.,  \cite{Friedman}), it  is defined in $B_b(K)$  (the space of the bounded and Borel measurable functions defined in $K$) by  
\begin{equation}
\label{e1.8}
\begin{array}{lll}
T^{K}_{\alpha}(t) \varphi(x)&=&\ds\E[\varphi(X_{\alpha}(t,x))\one_{\tau_x \ge t}]\\
\\
&=&\ds\int_{\{\tau_x \ge t\}}\varphi(X_{\alpha}(t,x))d\p,\quad\forall\;x\in K,
\end{array}
\end{equation}
where $\tau_x $ is the entrance  time in the complement of $K$, 
\begin{equation}
\label{e1.7}
\tau_x :=\inf\{t\ge 0:\;X_{\alpha}(t,x)\in K^c\},\quad\forall\;x\in K,
\end{equation}
and   $X_{\alpha}(t,x)$ is the solution to 
\begin{equation}
\label{e1.6}
dX_{\alpha}(t,x)= -\frac{1}{2}A^{\alpha}X_{\alpha}(t,x)dt+A^{(\alpha-1)/2}dW(t),\quad X(0,x)=x.
\end{equation}
Here $W(t)$ is a standard cylindrical Wiener process in $H$, defined in a filtered probability space $(\Omega,  \mathcal F, ({ \mathcal F}_t)_{t\geq 0}, \p)$.

The definition of   $T^{K}_{\alpha}(t)$ is similar to the one in \cite{Tal}, where the exit time from $ \oo{K}$,  $\widetilde{\tau}_x :=\inf\{t\ge 0:\;X_{\alpha}(t,x)\in  \oo{K} ^c\}$ was used instead of our $\tau_x$. 
In finite dimensions, if $K$ is the closure of a  bounded open set  with smooth boundary  the two definitions are equivalent, and $T^{K}_{\alpha}(t)$ is the semigroup associated to the realization of ${\mathcal L}_{\alpha}$ with Dirichlet boundary condition  (\cite[\S6.5]{Friedman}). Therefore,  a lot of regularity results, both interior and up to the boundary, are well known. 
In infinite dimensions,  interior regularity results were given in \cite{Tal}  for $\alpha >0$. We do not know regularity results up to the boundary, even in the case of very smooth bounded sets such as balls. 

Here we prove that    $\mu$  is a sub-invariant  measure for $T^{K}_{\alpha}(t)$. 
 Therefore, $T^{K}_{\alpha}(t)$ has a natural extension (still called $T^{K}_{\alpha}(t)$) to a contraction semigroup in $L^2(K,\mu)$.  The domain of its generator $L^{K}_{\alpha}$ consists of the range of the resolvent operator, 
\begin{equation}
\label{e1.10a}
R(\lambda, L^{K}_{\alpha})f =  \int_0^\infty e^{-\lambda t}T^{K}_{\alpha}(t) f dt,\quad f\in L^2(K,\mu),
\end{equation}
which is well defined for $\lambda >0$ since $T^{K}_{\alpha}(t)$ is a contraction semigroup. We prove that for each $\lambda >0$ and $f\in L^2(K,\mu)$, the function $\varphi := R(\lambda, L^{K}_{\alpha})f$ belongs to the above mentioned space $\oo{W}^{1,2}_{\alpha} (K,\mu)$, and satisfies the weak formulation of \eqref{e1.6a}. Therefore, $L^{K}_{\alpha} = M_{\alpha}$.

Our main  tool in the proof is  the approximating Feynman--Kac semigroup
\begin{equation}
\label{e1.11}
P^{\eps}_{\alpha}(t) \varphi(x)=\E\left[\varphi(X_{\alpha}(t,x))e^{-\frac1\eps\;\int_0^t  V(X_{\alpha}(s,x))ds}\right], 
\end{equation}
where $V$ is a (fixed) bounded continuous function that vanishes in $K$ and has positive values in  $K^c$. 
Its infinitesimal generator in $L^2(H,\mu)$ is the operator $M^{\eps}_{\alpha} : D(M^{\eps}_{\alpha}) = D(L_{\alpha})$ $\mapsto 
L^2(H, \mu)$, $M^{\eps }_{\alpha} \varphi = L_{\alpha}\varphi_\eps - \frac1\eps\,V\varphi$, and we  prove that for each $\varphi \in L^2(K,\mu)$, $t>0$, $\lambda >0$ we have
$$T^{K}_{\alpha}(t)\varphi = \lim_{\eps \to 0} (P^{\eps}_{\alpha}(t) \widetilde{\varphi})_{|K}, \quad R(\lambda, L^{K}_{\alpha})\varphi =  \lim_{\eps \to 0} (R(\lambda, M^{\eps}_{\alpha})\widetilde{\varphi})_{|K}$$
in $L^2(K,\mu)$, where $\widetilde{\varphi}$ is the null extension of $\varphi$ to the whole $H$.

Problem \eqref{e1.6a} is of interest for $\lambda =0$ too.  Using the fact that $D(L^{K}_{\alpha})$ is compactly embedded in $L^2(K,\mu)$, in Sect. 3.3 we prove that $0\in \rho(L^{K}_{\alpha})$ and that a Poincar\'e estimate holds in $\oo{W}^{1,2}_{\alpha}(K, \mu)$, for $\alpha\in (0, 1]$. Therefore, the supremum of $\sigma (L^{K}_{\alpha})$ is negative. 

These results are proved without additional assumptions on $K$. In particular, we do not require that $K$ is bounded, or that its boundary  is smooth. 

If the boundary of $K$ is suitably smooth,  it is possible to define surface integrals and traces at the boundary of functions in the Sobolev spaces $W^{1,2}_{\alpha}(K, \mu)$. 
Then we prove that the traces of the functions in $\oo{W}^{1,2}_{\alpha} (K,\mu)$ vanish. Therefore, the Dirichlet boundary condition in \eqref{e1.6a} is satisfied in the sense of the trace, and $T^{K}_{\alpha}(t)\varphi $ has null trace at the boundary for every $t>0$ and $\varphi\in L^2(K,\mu)$. 

Surface integrals for gaussian measures in Hilbert spaces are not a straightforward extension of the finite dimensional theory. To our knowledge the best reference is \cite[\S 6.10]{Bo}, where the Malliavin theory is presented. It deals with  
level surfaces of   smooth functions $g$ in a more general context than ours, since Souslin spaces $X$ are considered instead of Hilbert spaces. A part of the theory may be  simplified in our Hilbert setting, and moreover some of the smoothness assumptions on $g$ can be weakened. 
Therefore, we end the paper with an appendix  describing surface measures for level surfaces of suitably regular functions $g:H\mapsto \R$.

Several related important problems remain open, even for bounded $K$ with smooth boundary. Among them:
\begin{itemize}
\item[(a)]  While in finite dimensions 
$\varphi = R(\lambda, L^{K}_{\alpha})f$ is a strong solution to \eqref{e1.6a} 
and it belongs to $W^{2,2}(K,\mu)$ under reasonable assumptions on the boundary $\partial K$  (\cite{LMP}), 
in infinite dimensions we do not know whether $\varphi$ possesses second order derivatives in $L^2(K,\mu)$, even  if $K$ is the closed unit ball.  In fact, even in the case $\alpha =1$, the estimates found in \cite{DPGZ,Tal} 
are very bad  both near the boundary  and near $t=0$, and it is not clear how to use them  to get informations on the resolvent.
\item[(b)] We do not know whether $T^{K}_{\alpha}(t)$ is strong Feller in $K$ (i.e., it maps $B_b(K)$, the space of the bounded Borel functions  in $K$, to $C_b(K)$). This problem is open even for $K = \{ x\in H:\; |x| \leq 1\}$. 
\item[(c)] In finite dimensions, if  $\partial K$ is regular enough there are several characterizations of the space $\oo{W}^{1,2}_{\alpha} (K,\mu)$,  that coincides with $\oo{W}^{1,2}_{1}(K,\mu)$ for every $\alpha\in [0,1]$. The most obvious is the following: since $\mu$ is locally equivalent to the Lebesgue measure, $\oo{W}^{1,2}_{1} (K,\mu)$ coincides with the space  of the functions $f\in W^{1,2}_{1}(K,\mu)$ whose trace at the boundary vanishes. We do not know whether a  similar characterization holds in infinite dimensions. 
\end{itemize}

Referring to problem (a), in the  recent paper \cite{BDPT} a self-adjoint realization $L $ of $\mathcal L_1$ in $L^2(K,\mu)$ with Neumann boundary condition has been studied, in the case that $K$ is a convex   set with regular boundary. By means of a different (and better) approximation procedure, it has been proved that the resolvent $R(\lambda, L )$ maps  $L^2(K,\mu)$ into $W^{2,2}_{1}(K,\mu)$. 

Here we prove interior $W^{2,2}_{\alpha}$ regularity, for those $\alpha$ such that   Tr$[Q^{1-\alpha}]<\infty$. In this case  we show that 
  for every ball $B\subset K$ with positive distance from $\partial K$ and for every $\varphi \in D(L^{K}_{\alpha})$, the restriction $\varphi_{|B}$ belongs to $W^{2,2}_{\alpha}(B, \mu)$.


\section{Notation and preliminaries}


We denote by $\langle \cdot, \cdot\rangle $ and by $|\cdot|$ the scalar product and the norm in $H$. ${\mathcal L}(H)$ is the space of the linear bounded operators in $H$.

Let $Q$ be a symmetric (strictly) positive operator in ${\mathcal L}(H)$ with finite trace, and let $A:=Q^{-1}$. 
Accordingly, let  $\{e_k\}$ be an orthonormal basis in $H$ consisting of eigenfunctions of $Q$, i.e.
$$Qe_k=\lambda_k e_k,\;\;Ae_k = \frac{1}{\lambda_k}e_k, \quad \forall\;k\in \N .$$
We denote by $D_k$ the derivative in the direction of  $e_k$ and by $D$ the gradient of any differentiable function. Moreover we
 set $x_k=\langle x,e_k   \rangle$   for all $x\in H,\; k\in \N$. 

Throughout the paper we consider the $\sigma$-algebra $\mathcal{B}(H)$ of the Borel subsets of $H$ and the Gaussian measure with center $0$ and covariance $Q$  in $\mathcal{B}(H)$, denoting it by $\mu$.

An orthonormal basis of $L^2(H, \mu)$ consists of the Hermite polynomials. More precisely,  
for each $n\in \N \cup\{0\}$    let 
$$H_n(\xi) := (-1)^n n!^{-1/2} e^{\xi^2/2} D^n(e^{-\xi^2/2}), \quad \xi \in \R , $$
 be the usual normalized $n$-th Hermite polynomial. We denote by $\Gamma $ the set of all 
$\gamma : \N \mapsto \N \cup \{0\}$ such that  $\sum_{k=1}^\infty \gamma(k) <\infty$. For each $\gamma\in \Gamma$ let 
$$H_\gamma (x) := \prod_{k=1}^{\infty} H_{\gamma(k)}\bigg( \frac{x_k}{\sqrt{\lambda_k}} \bigg), \quad x\in H,$$ 
be the corresponding Hermite polynomial in $H$. Then, the  linear span ${\mathcal H}$ of all the Hermite polynomials $H_{\gamma}$ is dense in $L^2(H, \mu)$, and the linear span $\Lambda_0$ of the functions $H_{\gamma} \otimes e_h$, with $\gamma \in \Gamma$ and $h\in \N$,  is dense in the space $L^2(H,\mu; H)$ of all the (equivalence classes of) measurable functions $F:H\mapsto H$ such that $\int_{H}|F(x)|^2 \mu (dx)<\infty$. 

Other important dense subspaces of  $L^2(H, \mu)$ are the spaces  ${\mathcal E}_{\alpha}(H)$,  the linear spans  of the real and imaginary parts of the   functions $x\mapsto e^{i\langle x,  h \rangle} $, with $h\in D(A^{\alpha})$, $0\leq \alpha \leq 1$.

\subsection{Sobolev spaces over $H$}

We have the following integration   formula, 
\begin{equation}
 \int_{H} D_k\varphi \, d\mu =  \frac{1}{\lambda_k}\int_H x_k\varphi \,d\mu , \quad \varphi \in {\mathcal E}_{\alpha}(H) , \;  k\in \N . 
 \label{intpartiE}
 \end{equation}
It   may be extended to 
\begin{equation}
\int_{H} \langle D\varphi, G\rangle d\mu + \int_{H} \varphi \, \mbox{\rm div}\,G \,d\mu = \int_{H}\varphi  \langle x, AG(x)\rangle d\mu, \quad \varphi \in C^1_b(H), \; G\in \Lambda_0, 
 \label{intpartiG}
 \end{equation}
where div$\,G (x)=   \sum_{k=1}^{\infty}\langle DG(x), e_k\rangle$.  The  linear  operator $Q^{(1-\alpha)/2}D$ is well defined from ${\mathcal E}_{\alpha}(H) \subset L^2(H, \mu)$ to $ L^2(H,\mu; H)$, by 
 $$ Q^{(1-\alpha)/2}D\varphi = \sum_{k=1}^{\infty} \lambda_{k}^{(1-\alpha)/2}D_k\varphi  \,e_k.$$
Using   formula \eqref{intpartiG} it is easy to see that $ Q^{(1-\alpha)/2}D$ is closable. We still denote by $ Q^{(1-\alpha)/2}D$ its closure, and by $W^{1,2}_{\alpha}(H, \mu)$ the domain of the closure. (Note that for $\alpha=0$, $ Q^{1/2}D$ is nothing but the Malliavin derivative). $W^{1,2}_{\alpha}(H, \mu)$ is endowed with  the inner product  
\begin{equation}
\label{prodscal1,2}
\begin{array}{lll}
\langle \varphi,\psi  \rangle_{W_{\alpha}^{1,2}(H,\mu)}&=&\ds\int_H\varphi\psi \,d\mu
+\int_H\langle Q^{(1-\alpha)/2}D\varphi,Q^{(1-\alpha)/2}D\varphi  \rangle d\mu\\
\\
&=&\ds \int_H\varphi\psi \,d\mu
+\sum_{k=1}^\infty\int_H\lambda^{1-\alpha }_k D_k\varphi D_k\psi \, d\mu.
\end{array}
 \end{equation}
 So, $W^{1,2}_{\alpha}(H, \mu)$ is the completion of ${\mathcal E}_{\alpha}(H)$ in the norm associated to the scalar product \eqref{prodscal1,2}. It is also possible to characterize it through the Hermite polynomials. We have $\varphi\in W_{ \alpha}^{1,2}(H,\mu)$ iff
$$\sum_{\gamma\in \Gamma}  \sum_{h=1}^{\infty} \gamma_h \lambda_h^{-\alpha}\varphi_{\gamma}^2 <\infty$$
in which case the above sum is equal to $\int_H |Q^{(1-\alpha)/2}D\varphi|^2d\mu$. Indeed, the proof in \cite[Sect. 9.2.3]{DPZ3} for $\alpha =1$ works as well for any $\alpha \in [0,1)$.  
 
From this characterization it is clear that $W^{1,2}_{\alpha}(H, \mu) \subset  W^{1,2}_{0}(H, \mu)$ for every $\alpha \in (0, 1]$, with continuous embedding.

 Similarly, $W^{2,2}_{\alpha}(H, \mu)$  is the completion of ${\mathcal E}_{\alpha}(H)$ in the norm associated to the scalar product 
$$
\begin{array}{lll}
\langle \varphi,\psi  \rangle_{W_{ \alpha}^{2,2}(H,\mu)}&=&\ds\langle \varphi,\psi  \rangle_{W_{ \alpha}^{1,2}(H,\mu)}+\int_H \mbox{\rm Tr}\;[Q^{2-2\alpha}D^2\varphi D^2\psi] d\mu\\
\\
&=&\ds \langle \varphi,\psi  \rangle_{W_{ \alpha}^{1,2}(H,\mu)}+ \sum_{h,k=1}^\infty\int_H\lambda^{1-\alpha}_h\lambda^{1-\alpha}_k D_{h,k}\varphi D_{h,k}\psi \, d\mu.
\end{array}
$$

Next lemma is a consequence of   \cite[Lemma~5.1.12]{Bo} or \cite[Lemma~9.2.7]{DPZ3}.

 \begin{lem}
\label{emb}
There is $C>0$ such that 
$$\int_H |x|^2 \varphi(x)^2d\mu \leq C\|\varphi\|^2_{W^{1,2}_{0}(H, \mu)} , \quad \varphi \in  W^{1,2}_{0}(H, \mu).$$
 \end{lem}
 
Lemma \ref{emb}, together with \eqref{intpartiE}, yields the integration by parts formula in $W^{1,2}_{0}(H, \mu)$ (and hence, in all spaces $W^{1,2}_{\alpha}(H, \mu)$), 
\begin{equation}
 \int_{H} D_k\varphi \, \psi\, d\mu =  - \int_{H} \varphi \, D_k\psi\, d\mu  + \frac{1}{\lambda_k}\int_H x_k\varphi \,\psi \, d\mu , \quad \varphi, \;\psi  \in W^{1,2}_{0}(H, \mu), \;k\in \N . 
 \label{intparti}
 \end{equation}

For $0\leq \alpha \leq 1$ let $ T_{\alpha}(t)$ be the Ornstein-Uhlenbeck semigroup
\begin{equation}
\label{OUalpha}
 T_{\alpha}(t)\varphi(x) := \int_H \varphi(y){ \mathcal N}_{e^{-tA^{\alpha}/2} x, Q_t}(dy), \quad t>0, 
 \end{equation}
with 
$$Q_t := \int_0^t e^{-sA^{\alpha} }Q^{1-\alpha} ds = Q(I- e^{-tA^{\alpha} }).$$
$T_{\alpha}(t)$ is a Markov semigroup in $C_b(H)$, whose unique invariant measure is $\mu$. Its extension to $L^2(H, \mu)$ is a strongly continuous contraction semigroup, still denoted by 
$T_{\alpha}(t)$, whose  infinitesimal generator $L_{\alpha}$   is the closure of
${\mathcal L}_{\alpha} :{\mathcal E}_{\alpha}(H) \mapsto L^2(H, \mu)$. 

The domain of  $L_{\alpha}$ is continuously embedded in $W_{ \alpha}^{2,2}(H,\mu)$. 
Moreover,  for any $\varphi$, $\psi\in D(L_{\alpha})$ we have 
\begin{equation}
\label{e1.6aa}
\int_{H}L_{\alpha}\varphi\;\psi\,d\mu =-\frac12\;
\int_{H}\langle Q^{(1-\alpha)/2}D\varphi, Q^{(1-\alpha)/2}D\psi\rangle  d\mu.
\end{equation}

We refer to \cite[Ch.~9,~10]{DPZ3} for the proofs of   the above statements, and we add further properties of the spaces  $W_{ \alpha}^{1,2}(H,\mu)$ that will be used later. For each $\varphi\in L^1(H, \mu)$ we denote by 
$ \overline{\varphi}$ the mean value of $\varphi$, 
$$ \overline{\varphi} := \int_H \varphi\,d\mu .$$

\begin{prop}
\label{proprieta'}
Let  $0\leq \alpha \leq 1$. Then
\begin{itemize}
\item[(a)] A Poincar\'e estimate holds in $W_{ \alpha}^{1,2}(H,\mu)$, and precisely
\begin{equation}
\int_H (\varphi - \overline{\varphi})^2 d\mu \leq  \lambda_1^{\alpha} \int_H  |Q^{(1-\alpha)/2}D\varphi|^2 d\mu, 
\label{Poincare}
 \end{equation}
 where $\lambda_1$ is the maximum eigenvalue of $Q$.
 \item[(b)] The space  $W_{ \alpha}^{1,2}(H,\mu)$ is compactly embedded in $L^2(H, \mu)$ for $\alpha >0$.
 \end{itemize}
\end{prop}
\begin{proof}
A proof of statement (a) that follows the approach of Deuschel and Strook \cite{DS} is in \cite[Ch.10]{DPZ3}
  for $\alpha =1$. The same procedure works for $\alpha\in [0,1)$, since the key points of the proof still hold. Precisely, we have
\begin{itemize}
\item[(i)]
$|Q^{(1-\alpha)/2}DT^{\alpha}(t)\varphi|^2 \leq e^{-t/\lambda_1^{\alpha}}T^{\alpha}(t)(|Q^{(1-\alpha)/2}D\varphi|^2), \; \varphi\in C^1_b(H), \;t>0;$
\item [(ii)]
$\displaystyle \int_H \varphi L_{\alpha}\varphi \,d\mu  = -\frac{1}{2} \int_H |Q^{(1-\alpha)/2}D \varphi|^2d\mu , \quad \varphi \in D( L_{\alpha});$
\item[(iii)]
$\lim_{t\to \infty} T^{\alpha}(t)\varphi (x) = \overline{\varphi}, \quad \varphi\in {\mathcal E}_{\alpha}(H), \; x\in H.$
\end{itemize}
Once (i), (ii), (iii) are satisfied one can follow the proof of \cite[Prop.~10.5.2]{DPZ3} step by step. 
(ii) and (iii) follow from \cite[Prop.~10.2.3, Prop.~10.1.1]{DPZ3}. To check that (i) holds is easy and it is left to the reader. 
 
Statement (b) should be well known, however we give here a simple proof following \cite[Thm.~10.16]{DP} that concerns the case $\alpha =1$. We write every element $\varphi $ of $L^2(H, \mu)$ as $\varphi = \sum_{\gamma \in \Gamma} \varphi_{\gamma}H_{\gamma}$, with $\varphi_{\gamma} = \langle \varphi, H_{\gamma}\rangle$. We already  remarked  that $\varphi\in W_{ \alpha}^{1,2}(H,\mu)$ iff
$$\sum_{\gamma\in \Gamma}  \sum_{h=1}^{\infty} \gamma_h \lambda_h^{-\alpha}\varphi_{\gamma}^2 <\infty .$$
If a sequence $(\varphi ^{(n)})$ is bounded in $W_{ \alpha}^{1,2}(H,\mu)$, say $\|\varphi ^{(n)}\|_{W_{ \alpha}^{1,2}(H,\mu)}\leq K$ for each $n\in \N$, a subsequence $(\varphi ^{(n_k)})$ converges weakly in $  W^{1,2}_{ \alpha}(H,\mu)$ to a limit  $\varphi  $, that still satisfies 
$\|\varphi  \|_{W_{ \alpha}^{1,2}(H,\mu)}\leq K$. We shall show  that $\lim_{k\to \infty } \| \varphi ^{(n_k)} -\varphi\|_{L^2(H, \mu)} =0$. 

For each $N\in \N$, let $\Gamma_N =\{ \gamma \in \Gamma: \; \sum_{h=1}^{\infty} \gamma_h \lambda_h^{-\alpha} <\infty\}$. Then 
$$\begin{array}{l}
\displaystyle{ \int_H (\varphi ^{(n_k)} -\varphi)^2d\mu = \sum_{\gamma\in \Gamma_N} ( \varphi ^{(n_k)}_{\gamma} - \varphi_{\gamma})^2 
+ \sum_{\gamma\in \Gamma_N^c} ( \varphi ^{(n_k)}_{\gamma} - \varphi_{\gamma})^2 }
\\
\\
\leq \displaystyle{  \sum_{\gamma\in \Gamma_N} ( \varphi ^{(n_k)}_{\gamma} - \varphi_{\gamma})^2 + \frac{1}{N}\sum_{\gamma\in \Gamma} \sum_{h=1}^{\infty} \gamma_h \lambda_h^{-\alpha}( \varphi ^{(n_k)}_{\gamma} - \varphi_{\gamma})^2 }
\\
\\
\leq \displaystyle{  \sum_{\gamma\in \Gamma_N} ( \varphi ^{(n_k)}_{\gamma} - \varphi_{\gamma})^2  + \frac{(2K)^2}{N}.}
\end{array}$$
For $\eps >0$ fix $N\in \N$ such that $4K^2/N\leq \eps$. Since $\alpha >0$, then $\lim_{h\to \infty} \lambda_h^{-\alpha}  = +\infty$, so that the set $\Gamma_N$ has a finite number of elements. Since $ \varphi ^{(n_k)}$ converges weakly to $\varphi $ in $W_{ \alpha}^{1,2}(H,\mu)$, it converges weakly to $\varphi $ in $L^2(H, \mu)$; in particular $\lim_{h\to \infty}  \varphi ^{(n_k)}_{\gamma} = \varphi_{\gamma}  $ for each $\gamma \in \Gamma_N$. Therefore, for $k$ large enough we have $ \sum_{\gamma\in \Gamma_N} ( \varphi ^{(n_k)}_{\gamma} - \varphi_{\gamma})^2 \leq \eps$, and the statement follows.  \end{proof}

\subsection{Sobolev spaces over $K$}

Throughout the paper we assume that $K\subset H$ is a closed set with positive measure. To avoid trivialities, we assume that also  $K^c$ has positive measure.  

To treat  the Dirichlet problem \eqref{e1.6a} we introduce Sobolev spaces over $K$. 
We denote by $W^{1,2}_{\alpha} (K,\mu)$ the space of the functions $u: K  \mapsto \R$ that have an extension belonging to $W^{1,2}_{\alpha} (H,\mu)$, endowed with the standard inf norm. 
Moreover we  denote by $\oo{W}^{1,2}_{\alpha} (K,\mu)$ the subspace of $W^{1,2}_{\alpha} (K,\mu)$ consisting of the functions $u: K  \mapsto \R$ whose null extension to the whole $H$ belongs to the Sobolev space $W^{1,2}_{\alpha}(H,\mu)$. Therefore, 
$$\|u\|^{2}_{W^{1,2}_{\alpha} (K,\mu)} = \int_K u^2d\mu +  \int_K |Q^{(1-\alpha)/2}Du|^2d\mu, \quad u\in \oo{W}^{1,2}_{\alpha} (K,\mu),$$
 so that the 
$W^{1,2}_{\alpha} (K,\mu)$-norm in $\oo{W}^{1,2}_{\alpha} (K,\mu)$ is associated to the inner product
\begin{equation}
\label{prod.scal}
\langle u,v\rangle_{W^{1,2}_{\alpha}(K,\mu)} = \int_{ K} u\,v\,d\mu + \int_{ K} \langle  Q^{(1-\alpha)/2}Du,  Q^{(1-\alpha)/2}Dv \rangle \,d\mu .
\end{equation}
From the results of the next section it will be clear that such a space is not trivial, since it coincides with the domain 
of $(I-L^{K}_{\alpha})^{1/2}$.  Moreover, since $W^{1,2}_{\alpha}(H,\mu)$ is continuously embedded in $W^{1,2}_{0}(H,\mu)$, then 
$\oo{W}^{1,2}_{\alpha} (K,\mu)$ is continuously embedded in $\oo{W}^{1,2}_{0} (K,\mu)$, for every $\alpha \in (0, 1]$.

\subsection{The weak solution to \eqref{e1.6a}}
\label{sect:Diri}

The quadratic form $\mathcal {\mathcal Q}_{\alpha}$ associated to $\mathcal L_{\alpha}$, 
\begin{equation}
\label{Q}
{\mathcal Q}_{\alpha}(u,v) := \frac{1}{2}\int_{ K} \langle Q^{(1-\alpha)/2}Du, Q^{(1-\alpha)/2}Dv \rangle \,d\mu , \quad u,v\in  \oo{W}^{1,2}_{\alpha} (K,\mu),
\end{equation}
is continuous, nonnegative, and symmetric. Therefore, for every $\lambda >0$ and $f\in L^2( K,\mu)$ there exists a unique $\varphi \in \oo{W}^{1,2}_{\alpha} (K,\mu)$ such that 
\begin{equation}
\label{LM}
\lambda \int_{ K} \varphi \,v\,d\mu + \frac{1}{2}\int_{ K} \langle Q^{(1-\alpha)/2}D\varphi , Q^{(1-\alpha)/2}Dv \rangle \,d\mu = \int_{ K} f\,v\,d\mu, \quad \forall v\in \oo{W}^{1,2}_{\alpha} (K,\mu).
\end{equation}
The function $\varphi$ may be considered a weak solution to  \eqref{e1.6a}. 
Moreover, there exists a dissipative self-adjoint operator $M_{\alpha}$ in $L^2( K,\mu)$ such that $\varphi = R(\lambda, M_{\alpha})f$. Like all dissipative self-adjoint operators in Hilbert spaces, $M_{\alpha}$ is the infinitesimal generator of an analytic contraction semigroup, and several properties of $M_{\alpha}$  follow. See e.g. \cite[Ch.~6]{Kato}.


 \section{The Dirichlet semigroup}
 

In this section we give an explicit representation formula for the semigroup generated by  the operator $M_{\alpha}$ defined in  section \ref{sect:Diri}, through the approximation procedure described in the introduction. Moreover we show some properties of the semigroup and of its generator. 

\subsection{The approximating semigroups}

We fix once and for all a   function $V\in C_b(H)$ such that 
 $$V(x)=0, \;x\in K, \quad V(x)>0, \;x\in K^c . $$
For   $\eps >0$ let $P^{\eps}_{\alpha}(t)$ be defined by \eqref{e1.11}.

\begin{prop}
\label{p2.2}
For any  $\varphi\in C_b(H)$ we have
\begin{equation}
\label{e2.4}
\int_H(P^{\eps}_{\alpha}(t)\varphi(x))^2\mu(dx)\le
\int_H\varphi^2(x)\mu(dx).
\end{equation}
Consequently, $P^{\eps}_{\alpha}(t)$ is uniquely extendable to a $C_0$-semigroup in $L^2(H,\mu)$ which  we shall denote by the same symbol.
\end{prop}
\begin{proof} We have in fact, by the H\"older inequality
$$
(P^{\eps}_{\alpha}(t)\varphi(x))^2\le\E  \left(\varphi^2(X_{\alpha}(t,x))e^{-\frac2\eps\;\int_0^tV(X_{\alpha}(s,x))ds}\right)\le  T_{\alpha}(t)(\varphi^2)(x),
$$
where $ T_{\alpha}(t)$ is the Ornstein-Uhlenbeck semigroup defined in \eqref{OUalpha}.
Since $\mu$ is invariant for $T_{\alpha}(t)$, then
\begin{align}
 \int_H (P^{\eps}_{\alpha}(t)\varphi(x))^2\mu(dx)\le \int_H T_{\alpha}(t)(\varphi^2)(x)\mu(dx)= \int_H\varphi^2(x)\mu(dx).  
\end{align}
 \end{proof}

We denote  by $M^{\eps}_{\alpha}$ the infinitesimal  generator of  $P^{\eps}_{\alpha}(t)$  in $L^2(H,\mu)$ and we want to show
 that $M^{\eps}_{\alpha}=L_{\alpha} -\frac1\eps \,V .$ To this aim, for $\lambda>0$ and $f\in L^2(H, \mu)$ we consider the resolvent equation
\begin{equation}
\label{e2.3}
\lambda \varphi_\eps-L_{\alpha}\varphi_\eps+\frac1\eps\;V\varphi_\eps = f.
\end{equation}

  \begin{prop}
\label{p2.4}
Let $\lambda>0$, $\eps >0$, and $f\in L^2(H,\mu)$. Then  equation \eqref{e2.3} has a unique solution   $\varphi_\eps \in D(L_{\alpha})$, and the following estimates hold.
\begin{equation}
\label{e2.7}
\int_H\varphi^2_\eps d\mu\le \frac1{\lambda^2}\;\int_Hf^2 d\mu,
\end{equation}
\begin{equation}
\label{e2.8}
\int_H|Q^{(1-\alpha)/2}D\varphi _\eps|^2 d\mu\le \frac2{\lambda}\;\int_Hf^2 d\mu,
\end{equation}
\begin{equation}
\label{e2.9}
\int_{K^c}  V \varphi^2_\eps  d\mu\le \frac\eps{\lambda }\;\int_Hf^2 d\mu.
\end{equation}
\end{prop}
\begin{proof} Fix $\lambda >0$ and $\eps >0$. Since $L_{\alpha}$ is maximal dissipative and  $\varphi\to\frac1\eps\, V\varphi$ is bounded and monotone increasing in $L^2(H,\mu)$, it follows by standard arguments that  the operator
$$
D(L_{\alpha})\mapsto L^2(H,\mu), \quad \varphi \mapsto L_{\alpha} \varphi -\frac1\eps\;V\varphi,
$$
is maximal dissipative. So, equation \eqref{e2.3}  has a unique solution $\varphi_\eps\in D(L_{\alpha})$, that satisfies \eqref{e2.7}.

Multiplying both sides of \eqref{e2.3} by $\varphi_\eps$, integrating over $H$  and taking into account \eqref{e1.6aa}
yields
\begin{equation}
\label{e2.10a}
\lambda\int_H \varphi_\eps^2d\mu+\frac12\;
 \int_H|Q^{(1-\alpha)/2}D\varphi_\eps|^2d\mu+\frac1\eps\;\int_{K^c} V \varphi_\eps^2d\mu =
\int_H f\varphi_\eps d\mu.
\end{equation}
The inequality $\lambda\int_H|\varphi_\eps|^2d\mu \leq \int_H f\varphi_\eps d\mu$ yields again 
 \eqref{e2.7}. The inequality
$$
\frac12\;
 \int_H|Q^{(1-\alpha)/2}D\varphi_\eps|^2d\mu\le \int_H f\varphi_\eps d\mu
$$
implies \eqref{e2.8}, using the H\"older inequality in the right-hand side and then \eqref{e2.7}. 
The inequality
$$
\frac1\eps\;\int_{K^c}V  \varphi_\eps^2d\mu \le
\int_H f\varphi_\eps d\mu  
$$
implies \eqref{e2.9}, using again the H\"older inequality in the right-hand side and then  \eqref{e2.7}.  \end{proof}

\begin{prop}
\label{p2.3}
Let  $M^{\eps}_{\alpha}$ be the infinitesimal generator of $P^{\eps}_{\alpha}(t)$. Then $D(M^{\eps}_{\alpha} ) = D(L_{\alpha})$ and
\begin{equation}
\label{e2.5}
M^{\eps}_{\alpha}\varphi=L_{\alpha}\varphi-\frac1\eps\;V\varphi,\quad\forall\;\varphi\in D(L_{\alpha}).
\end{equation}
\end{prop}
\begin{proof} Let us   show that 
$D(L_{\alpha})\subset D(M^{\eps}_{\alpha})$, and that \eqref{e2.5} holds.  

First, let $\varphi\in D(L_{\alpha})\cap C_b(H)$. For $x\in H$, $h>0$  we have
\begin{equation}
\label{numero}
P^\eps_h\varphi(x)-\varphi(x)
= T_{\alpha}(h)\varphi(x)-\varphi(x) + \E\left[\left(e^{-\frac1\eps\;\int_0^h V(X_{\alpha}(r,x))dr}-1\right)\varphi(X_{\alpha}(h,x))\right].
\end{equation}
We recall that, since $A^{\alpha}$ is self-adjoint, $X_{\alpha}(\cdot,x)$  possesses a.s.  continuous paths (\cite{K,T}).  Therefore the functions $r\mapsto \varphi(X_{\alpha}(r,x))$ and $r\mapsto V(X_{\alpha}(r,x))$ are continuous a.s.
Dividing both sides of \eqref{numero} by $h$ and letting $h\to 0$, we obtain $\lim_{h\to 0}
(P^\eps_h\varphi -\varphi )/h = L_{\alpha}\varphi - V\,\varphi/\eps$ pointwise and (by dominated convergence) in $L^2(H,\mu)$, so that $\varphi\in D(M^{\eps}_{\alpha})$ and  \eqref{e2.5} holds.

Let now $\varphi\in D(L_{\alpha})$, and  let 
 $(\varphi_n)$ be a sequence of functions in $ \mathcal E_{\alpha}(H)$ that converges to $\varphi$ in $D(L_{\alpha})$. 
Then, $\varphi_n\to \varphi$ in $L^2(H, \mu)$, so that $\frac1\eps\;V \varphi_n \to \frac1\eps\; V \varphi$ in $L^2(H, \mu)$, moreover $L_{\alpha}\varphi_n\to L_{\alpha}\varphi $  in $L^2(H, \mu)$. It follows that $M^{\eps}_{\alpha} \varphi_n\to M^{\eps}_{\alpha} \varphi$ in $L^2(H, \mu)$, and since $M^{\eps}_{\alpha}$ is closed, then $\varphi \in D(M^{\eps}_{\alpha})$ and  \eqref{e2.5} holds.\medskip

The other inclusion $D(M^{\eps}_{\alpha})\subset D(L_{\alpha})$ is immediate.  Indeed, for any  $\varphi\in D(M^{\eps}_{\alpha})$   set $f=\lambda \varphi-M^{\eps}_{\alpha} \varphi$, and let $ \varphi_\eps$ be the solution of \eqref{e2.3}. Then
$ \varphi_\eps\in D(L_{\alpha})\subset D(M^{\eps}_{\alpha})$, so that $(\lambda-M^{\eps}_{\alpha})^{-1}f=\varphi_\eps=\varphi$ which implies that $\varphi\in D(L_{\alpha})$. 
\end{proof}


\begin{rem}
From the very beginning, one would be tempted to replace the continuous function $V$ by $\one_{K^c}$ in the definition of $M_{\eps}$. But with this choice the proof  of Proposition \ref{p2.3} does not work. Indeed, it is not obvious that $(P^\eps_h\varphi -\varphi )/h$ converges as $h\to 0$ for any $\varphi \in C_b(H)\cap D(L_{\alpha})$, if $x\in \partial K$, 
because the function $r\mapsto \one_{K^c}(X_{\alpha}(r,x))$ could be discontinuous at $r=0$. If $\mu(\partial K)=0$ this difficulty is not relevant, since we are interested in $L^2$ convergence rather than in pointwise convergence. However, we prefer to make no further assumptions on $\partial K$ in this first part of the paper. 
\end{rem}

\subsection{Identification of $T^{K}_{\alpha}(t)$}

 Let $T^{K}_{\alpha}(t)$, $P^{\eps}_{\alpha}(t)$  be defined by \eqref{e1.8}, \eqref{e1.11} respectively.

 \begin{prop}
\label{p2.1}
For any $\varphi\in B_b(H)$, $t>0$, and for any $x\in K$ we have
\begin{equation}
\label{e2.1}
\lim_{\eps\to 0}P^{\eps}_{\alpha}(t)\varphi(x)=T^{K}_{\alpha}(t)\varphi_{|K}(x).
\end{equation}
Moreover $T^{K}_{\alpha}(t)$  is a semigroup of linear bounded operators in $B_b(K)$.
\end{prop}
\begin{proof}  Let $t>0$, $x\in  K$. Then
$$
\{\tau_x^K\ge t\}=\{\omega\in\Omega:\;X_{\alpha}(s,x)\in   K,\;\forall\;s\in[0,t)\}
$$
and
$$
\{\tau_x^K<t\}=\{\omega\in\Omega:\;\exists\;s_0\in (0,t):\; X_{\alpha}(s_0,x)\in K^c\}
$$
Then we have
$$
P^{\eps}_{\alpha}(t)\varphi(x)=\int_{\{\tau_x^K\ge t\}}\varphi(X_{\alpha}(t,x))d\p+\int_{\{\tau_x^K<t\}}\varphi(X_{\alpha}(t,x))e^{-\frac1\eps\;\int_0^t  V(X_{\alpha}(s,x))ds}d\p
$$
In view of the dominated convergence theorem, to prove the statement  it is enough to show that
\begin{equation}
\label{e2.2}
\lim_{\eps\to 0} e^{-\frac1\eps\;\int_0^t V(X_{\alpha}(s,x))ds}=0,
\end{equation}
for a.a. $\omega$ such that $\tau_x^K(\omega)<t$.

We already mentioned that  $X_{\alpha}(\cdot,x)$  possesses a.s.  continuous paths.   
 Let  $\omega\in \Omega$ be such that $X_{\alpha}(\cdot,x)(\omega)$ is continuous. If $\tau_x^K(\omega)<t$,  there exist $s_0 <t$, $\delta >0$ (depending on $\omega$) such that
$$
 X_{\alpha}(s ,x)   \in K^c,\quad\forall\; s\in [s_0-\delta ,s_0+\delta].
$$
Since $V$ is continuous and it has positive values in $K^c$, then
$$c:= \inf \{ V(X(s,x)):\; s\in [s_0-\delta ,s_0+\delta]\} >0.$$
It follows that
$$
 e^{-\frac1\eps\;\int_0^t V(X_{\alpha}(s,x))ds}\le  e^{-\frac{2c}{\eps}\;\delta }\to 0,\;\mbox{\rm as}\;\eps\to 0.
$$
So,  \eqref{e2.2} holds. The last statement is straightforward.
 \end{proof}

In the next proposition we show that $\mu$ is sub-invariant for $T^{K}_{\alpha}(t)$. 
 We   use the  following notation.
For each $\varphi\in B_b(K)$ we set
$$
\widetilde{\varphi}(x)=\left\{\begin{array}{l}
\varphi(x),\quad\mbox{\rm if}\;x\in K,\\
0, \quad\mbox{\rm if}\;x\notin K.
\end{array}\right.
$$

 \begin{prop}
\label{p3.1}
For any $\varphi\in B_b(K)$, $t>0$,   we have
\begin{equation}
\label{e3.2}
\int_K(T^{K}_{\alpha}(t)\varphi(x))^2\mu(dx)\le \int_K\varphi^2(x)\mu(dx).
\end{equation}
Consequently, $T^{K}_{\alpha}(t)$ can be uniquely extended to a $C_0$ semigroup of contractions in $L^2(K,\mu)$.
\end{prop}
\begin{proof} By the H\"older inequality we have for all $x\in K$
$$
(T^{K}_{\alpha}(t)\varphi(x))^2\le \E[\varphi^2(X_{\alpha}(t,x))\one_{\tau_x^K\ge t}]
\le \E[\widetilde{\varphi}^2(X_{\alpha}(t,x))\one_{\tau_x^K\ge t}]\le T_{\alpha}(t)(\widetilde{\varphi}^2)(x) .
$$
Since $\mu$ is invariant for $T_{\alpha}(t)$, it follows that
 $$
 \begin{array}{l}
\ds \int_K(T^{K}_{\alpha}(t)\varphi(x))^2d\mu \le \int_KT_{\alpha}(t)(\widetilde{\varphi}^2)(x)d\mu \\
 \\
 \ds\le \int_HT_{\alpha}(t)(\widetilde{\varphi}^2)d\mu \le\int_HT_{\alpha}(t)(\widetilde{\varphi}^2)d\mu \le \int_H\widetilde{\varphi}^2d\mu =\int_K\varphi(x)^2d\mu .
 \end{array}
 $$
 The conclusion follows. \end{proof}

 We shall denote  by $L^{K}_{\alpha}$ the infinitesimal generator of $T^{K}_{\alpha}(t)$ in $L^2(K,\mu)$.

\begin{prop}
\label{p3.2}
For any $f\in L^2(K,\mu)$ and $t>0$ we have
\begin{equation}
\label{e3.4c}
\lim_{\epsilon\to 0} (P^{\eps}_{\alpha}(t) \widetilde{ f})_{|K} = T^{K}_{\alpha}(t) f,\quad \mbox{\rm in}\;L^2(K,\mu)
\end{equation}
and, for $\lambda >0$,
\begin{equation}
\label{e3.4}
\lim_{\epsilon\to 0} (R(\lambda, M^{\eps}_{\alpha} ) \widetilde{ f})_{|K}  =(\lambda-L^{K}_{\alpha})^{-1}f,\quad\mbox{\rm in}\; L^2(K,\mu). 
\end{equation}
 \end{prop}
\begin{proof} Let $f\in C_b(H)$. By Proposition \ref{p2.1}, $P^{\eps}_{\alpha} f $ converges pointwise to $T^{K}_{\alpha}(t) f$ in $K$. Moreover, 
$| (P ^{\eps}_{\alpha}(t)f)(x)| \leq \|f\|_{\infty}$, $|(T^{K}_{\alpha}(t)f)(x)| \leq \|f\|_{\infty}$ for each $x\in K$ and $t>0$.  By dominated convergence, 
$\lim_{\eps \to 0} \| P^{\eps}_{\alpha}(t) f - T^{K}_{\alpha}(t) f\|_{L^2(K,\mu)} =0$. 

Let now $f\in L^2(K,\mu)$. Since  $C_b(H)$ is dense in $L^2(H,\mu)$, there is a sequence 
 $(f_n)\subset C_b(H)$   such that
$$
\|\widetilde{f}-f_n\|_{L^2(H,\mu)}\le \frac1n,\quad\forall\;n\in \N.
$$
Then we have
$$
\begin{array}{l}
\|T^{K}_{\alpha}(t)  f-P^{\eps}_{\alpha}(t)\widetilde{ f}\|_{L^2(K,\mu)} \le
\|T^{K}_{\alpha}(t) ( f - f_n)\|_{L^2(K,\mu)}\\
\\
 +\|T^{K}_{\alpha}(t) f_n- P^{\eps}_{\alpha}(t) f_n \|_{L^2(K,\mu)}+\| P^{\eps}_{\alpha}(t) (f_n - \widetilde{f})  \|_{L^2(K,\mu)}
\\
\\
 \le \ds \frac2n+\|T^{K}_{\alpha}(t)  f_n- P^{\eps}_{\alpha}(t) \widetilde{f}_n \|_{L^2(K,\mu)},\quad\forall\;n\in \N , 
\end{array}
$$
and \eqref{e3.4c} follows.

To prove \eqref{e3.4},  we use the identity  (in $L^2(H,\mu)$)
$$R(\lambda, M^{\eps}_{\alpha} ) \widetilde{ f}  = \int_{0}^{\infty} e^{-\lambda t} P^{\eps}_{\alpha}(t)\widetilde{ f} \, dt .$$ 
Taking the restrictions to $K$ of both sides and using \eqref{e3.4c} we   obtain
$$\lim_{\epsilon\to 0} (R(\lambda, M^{\eps}_{\alpha} ) \widetilde{ f})_{|K}  =  \int_{0}^{\infty} e^{-\lambda t} T^{K}_{\alpha}(t) f\,  dt  , $$ 
which coincides with \eqref{e3.4}. 
\end{proof}

\begin{thm}
\label{Identificazione}
For every $\lambda >0$ and $f\in L^2( K,\mu)$, the function $\varphi:=R(\lambda, L^{K}_{\alpha})f$ belongs to $\oo{W}^{1,2}_{\alpha}(K,\mu)$ and satisfies \eqref{LM}.  Therefore, $T^{K}_{\alpha}(t)$ is the semigroup generated by $M_{\alpha}$ in $L^2( K,\mu)$. 
\end{thm}
\begin{proof}  For $\eps >0$ define $\varphi_{\eps} := R(\lambda, M^{\eps}_{\alpha} ) \widetilde{ f} $. 
By Proposition \ref{p2.3}, $\varphi_{\eps}$ is the solution to \eqref{e2.3}, with $f$ replaced by $ \widetilde{ f} $. 
By Proposition \ref{p2.4}, the $W^{1,2}_{\alpha}(H, \mu)$-norm of 
$\varphi_{\eps}$ is bounded by a constant independent of $\eps$. 
Therefore, there is a sequence $\eps_k\to 0$ such that $ \varphi_{\eps_k} $ converges 
weakly in $W^{1,2}_{ \alpha}(H, \mu)$ to a function $\Phi$. Let us prove that   $\Phi = \widetilde{\varphi}$. 

For every $\psi \in L^2(K, \mu)$ we have 
$$\int_K \Phi \psi \,d\mu =  \lim_{k\to \infty}\int_H  \varphi_{\eps_k}\widetilde{\psi}\,d\mu =  \lim_{k\to \infty}\int_K  \varphi_{\eps_k}\psi\,d\mu 
= \int_K \varphi \psi \,d\mu $$
since, by Proposition \ref{p3.2},  $\lim_{\eps \to 0} \| \varphi_{\eps |K} - \varphi\|_{L^2( K,\mu)} = 0$. Then, $\Phi_{|K} = \varphi$. 

Moreover, 
$$\int_{K^c} \Phi^2V\,d\mu = \int_H  \Phi\cdot \Phi V\one_{K^c}  d\mu = \lim_{k\to \infty}\int_H  \varphi_{\eps_k}\Phi V\one_{K^c}  d\mu ,$$
and by estimate  \eqref{e2.9} and the H\"older inequality we have 
$$\bigg| \int_H  \varphi_{\eps_k}\Phi V\one_{K^c}  d\mu \bigg| \leq  \bigg( \int_{K^c}  \varphi_{\eps_k}^2V\, d\mu \bigg)^{1/2} 
 \bigg( \int_{K^c} \Phi^2V\,d\mu  \bigg)^{1/2} \to 0 \quad \mbox{\rm as}\;k\to \infty .$$
It follows that  $\Phi_{|K^c} =0$. 
Therefore, 
  $\Phi = \widetilde{\varphi}\in W^{1,2}_{\alpha}(H, \mu)$, that is $\varphi \in \oo{W}^{1,2}_{\alpha}(K,\mu)$.

For every $v\in \oo{W}^{1,2}_{\alpha}(K,\mu)$ and $k\in \N$ we have (since $\int_H  V \varphi_{\eps _k} \widetilde{v}\,d\mu = 0$)
$$\lambda \int_H \varphi_{\eps _k} \widetilde{v}\,d\mu + \frac{1}{2} \int_H \langle Q^{(1-\alpha)/2}D\varphi_{\eps _k}, Q^{(1-\alpha)/2}D\widetilde{v} \rangle d\mu   = \int_H fv \,d\mu , $$
and  letting $k\to \infty$ we obtain 
$$\lambda \int_H \widetilde{\varphi} \widetilde{v} \,d\mu + \frac{1}{2} \int_H \langle Q^{(1-\alpha)/2}D\widetilde{\varphi} , Q^{(1-\alpha)/2}D\widetilde{v} \rangle d\mu   = \int_H f\widetilde{v} \,d\mu , $$
so that $\varphi$ satisfies \eqref{LM}, and the statement follows. 
\end{proof}

 \subsection{Consequences}
 
 We list here some consequences of the results of this section, that hold for every $\alpha \in [0,1]$.
 
 \begin{itemize}
 \item[(i)] $T^{K}_{\alpha}(t)$ is an analytic semigroup in $L^p( K,\mu)$ for every $p\in (1, \infty)$. 
 \item[(ii)] The space $ \oo{W}^{1,2}_{\alpha}(K,\mu)$ coincides with the domain of $(I-L^{K}_{\alpha})^{1/2}$. 
 \item[(iii)] For each $f\in L^2( K,\mu)$ we have
 $$\int_K | Q^{(1-\alpha)/2}DT^{K}_{\alpha}(t)f|^2\mu(dx)\le\frac1{\sqrt t}\;\int_K f^2(x)\mu(dx),\quad t>0.$$
 \end{itemize}

These statements follow in a standard way from the fact that the infinitesimal generator $L^K_{\alpha}$ of $T^{K}_{\alpha}(t)$  is the operator associated to the symmetric quadratic form $\mathcal Q_{\alpha}$ defined in \eqref{Q}, and that it is dissipative.   

Less standard consequences are a Poincar\'e inequality in the space   $\oo{W}^{1,2}_{\alpha}(K,\mu)$ and the invertibility of $L^K_{\alpha}$ for $\alpha >0$, proved in the next proposition.

\begin{prop}
\label{Pr:poinc}
For $\alpha \in (0,1]$ the spaces $\oo{W}^{1,2}_{\alpha}(K,\mu)$ and $D(L^{K}_{\alpha})$ are compactly embedded in $ L^2(K,\mu)$. Moreover $0\in \rho (L^{K}_{\alpha})$, and a Poincar\'e inequality holds in $\oo{W}^{1,2}_{\alpha}(K,\mu)$, 
 $$ \|u\|_{L^2( K,\mu)} \leq C \int_K | Q^{(1-\alpha)/2}Du|^2\,d\mu, \quad u\in \oo{W}^{1,2}_{\alpha}(K,\mu).$$
\end{prop}
\begin{proof} 
Since the embedding $W^{1,2}_{\alpha}(H, \mu)\subset L^2(H, \mu)$ is compact by Proposition \ref{proprieta'}(b), the embedding $\oo{W}^{1,2}_{\alpha}(K,\mu) \subset L^2(K,\mu)$ is compact too. 
Indeed,  a sequence $u_n$ is bounded in $\oo{W}^{1,2}_{\alpha}(K,\mu)$ iff  the sequence $\widetilde{u}_n$ is bounded in $W^{1,2}_{\alpha}(H, \mu)$. In this case, there is a subsequence of $   \widetilde{u}_n $ that converges  to a function $v\in L^2(H, \mu)$. Therefore, a subsequence of $u_n$ converges to the restriction  $v_{|K}$, in $L^2(K,\mu)$. 

Since the domain  $D(L^{K}_{\alpha})$ is continuously embedded in $\oo{W}^{1,2}_{\alpha}(K,\mu)$, it is compactly embedded in $L^2( K,\mu)$. Therefore, the spectrum of $L^K_{\alpha}$ consists of (nonpo\-sitive) eigenvalues. Let us prove that $0$ is not an eigenvalue. 

Let $u\in D(L^{K}_{\alpha})$ be such that $L^{K}_{\alpha}u=0$. Then 
$$ 0 = \int_K u L^{K}_{\alpha}u\,d\mu = -\frac{1}{2}\int_K |Q^{(1-\alpha)/2}Du|^2\,d\mu = -\frac{1}{2}\int_H |Q^{(1-\alpha)/2}D\widetilde{u}|^2d\mu ,$$
and by the Poincar\'e inequality in $W^{1,2}_{\alpha}(H,\mu)$ (Proposition \ref{proprieta'}(a)) we have
$$\int_{H} (\widetilde{u} - \int_H\widetilde{u}\, d\mu)^2d\mu =0. $$
So, $\widetilde{u}$ is constant a.e. in $H$, but since it vanishes  in $K^c$, whose measure is positive,  then it vanishes a.e. in $H$. Therefore, $u=0$. 
 
This implies that the seminorm   $u\mapsto  (\int_K |Q^{(1-\alpha)/2}Du|^2\,d\mu)^{1/2}$ is in fact an equivalent norm in $\oo{W}^{1,2}_{\alpha}(K,\mu)$, that is, a Poincar\'e inequality holds in $\oo{W}^{1,2}_{\alpha}(K,\mu)$. Indeed, since $-L^{K}_{\alpha}$ is invertible, also $(-L^{K}_{\alpha})^{1/2}$ is invertible, so that the seminorm $u\mapsto \|(-L^{K}_{\alpha})^{1/2}u\|_{L^2(K,\mu)} = 
\frac{1}{2}\int_K |Q^{(1-\alpha)/2}Du|^2\,d\mu$ is an equivalent  norm in $D((-L^{K}_{\alpha})^{1/2}) = \oo{W}^{1,2}_{\alpha}(K,\mu)$; in other words there is $C>0$ such that  
$\|u\|_{L^2(K,\mu)} \leq C \int_K |Q^{(1-\alpha)/2}Du|^2\,d\mu$ for every $u\in \oo{W}^{1,2}_{\alpha}(K,\mu)$. 
\end{proof}

 \section{Interior regularity} 

 In this section we  prove an interior regularity result for  the solution to \eqref{e1.6a} for $ \alpha <1$.  We use the following lemma.

\begin{lem}
\label{prodotto}
 For every $\varphi \in D(L_{\alpha})$ and for every $\beta\in {\mathcal E}_{\alpha}(H)$, the product $\varphi \beta$ belongs to the domain of   $L_{\alpha}$, and
$$ L_{\alpha}(\varphi\beta) = \beta L_{\alpha}\varphi + \varphi L_{\alpha}\beta + \langle Q^{1-\alpha}D\varphi, D\beta \rangle.$$
\end{lem}
 \begin{proof}  Since ${\mathcal E}_{\alpha}(H)$ is dense in $D(L_{\alpha})$, there is a sequence $(\varphi_n)\subset {\mathcal E}_{\alpha}(H)$ that converges to $\varphi$ in $D(L_{\alpha})$. For every $n$,     $\beta  \varphi_n$ is still in ${\mathcal E}_{\alpha}(H)$, hence it belongs to $D(L_{\alpha})$ and the statement follows easily.  \end{proof}

\begin{prop}
\label{RegPalla}
Assume that 
$$\mbox{\rm Tr}\;Q^{1-\alpha} = \sum_{k=1}^{\infty}\lambda_k^{1-\alpha} <\infty.$$
Then for every $y\in \oo{K}$ and   $r>0$ such that dist$(B(y, r),$ $ \partial K)>0$, the restriction to $B(y, r)$ of the solution $\varphi $ to \eqref{e1.6a} belongs to $W^{2,2}_{\alpha}(B(y, r), \mu)$. 
\end{prop}
 \begin{proof}  It is enough to prove that the statement holds for $y\in D(A^{\alpha/2})$. Indeed, since $D(A^{\alpha/2})$ is dense in $H$, for each $y\in \oo{K}$ and   $r>0$ such that dist$(B(y, r),$ $ \partial K)>0$ there are $y_1\in \oo{K}\cap D(A^{\alpha/2})$ and $r_1>r$ such that $B(y, r)\subset B(y_1, r_1)$ and dist$(B(y_1, r),$ $ \partial K)>0$. 
 
So, let $y\in   D(A^{\alpha/2} )$ and let $r_1>r$ be such that the ball $B(y, r_1)$ is contained in $\oo{K}$. Let $\rho :\R\mapsto [0,1]$ be a $C^2$ function  such that 
$$\rho (\xi)=1, \; \xi \leq r^2, \quad \rho(\xi )=0, \;\xi\geq r_1^2, $$ 
and define a cutoff function $\theta$ by 
$$\theta(x) := \rho( |x-y|^2), \quad x\in H.$$
Our aim is to show that the product $\widetilde{\varphi} \theta$ belongs to $ W^{2,2}_{\alpha}(H, \mu)$. Since  the restriction to $B(y, r)$ of $\widetilde{\varphi} \theta$ coincides with  the restriction  to $B(y, r)$ of $\varphi$, the statement will follow. 
 
The proof is in three steps. As a first step, we show that $\theta\in D(L_{\alpha})$. Then we show that  $ \varphi_{\eps} \theta $ belongs to $D(L_{\alpha})$ for every $\eps >0$, where $ \varphi_{\eps} = R(\lambda, M^{\eps}_{\alpha} )\widetilde{f}$. Eventually, we prove that $\widetilde{\varphi} \theta \in W^{2,2}_{\alpha}(H, \mu)$. 
 
\vspace{3mm}
\noindent {\em First step: $\theta\in D(L_{\alpha})$.}
We approach each $x\in H$ by the sequence $x_n = \sum_{k=1}^{n}\langle x, e_k \rangle e_k$, and we consider the sequence of functions
$$\theta_n(x):=    \rho (|x_n-y_n|^2), \quad x\in H, \;n\in \N.$$
Each of them belongs to $D(L_{\alpha})$. This is because it depends only on the first $n$ coordinates, it is bounded and it has bounded first and second order derivatives, and in finite dimensions the inclusion  $C^2_b(H) \subset D(L_{\alpha}) $  holds. Therefore, it is easy to see that 
there exists the limit $\lim_{t\to 0}( T_{\alpha}(t) \theta_n - \theta_n)/t = L_{\alpha} \theta_n$ in $L^2(H, \mu)$, where
\begin{equation}
\label{exp}
\begin{array}{lll}
L_{\alpha} \theta_n (x) & =&   \rho' (|x_n-y_n|^2) \displaystyle{\bigg( \sum_{k=1}^{n} \lambda_k^{1-\alpha} -   \sum_{k=1}^{n}\lambda_k^{-\alpha}  \langle x , e_k\rangle \langle x-y, e_k\rangle   \bigg)}
\\
\\
&  + & 2\rho '' (|x_n-y_n|^2)\langle Q^{1-\alpha}(x_n-y_n), x_n-y_n\rangle .
\end{array} 
\end{equation}

Letting $n\to \infty$, $ \rho' (|x_n-y_n|^2)$ and $\rho '' (|x_n-y_n|^2)\langle Q^{1-\alpha}(x_n-y_n), x_n-y_n\rangle $ converge in $L^2(H, \mu)$ to $\rho'(|x-y|^2)$ and to $\rho '' (|(x -y |^2)\langle Q^{1-\alpha}(x -y ), x -y \rangle $, respectively,  by dominated convergence. The sum $ \sum_{k=1}^{n}\lambda_k^{-\alpha}  \langle x , e_k\rangle \langle x-y, e_k\rangle $ converges too. Indeed, for $p<q\in \N$ we have
$$\begin{array}{l}
\ds{ \| \sum_{k=p}^{q}\lambda_k^{-\alpha}  \langle x , e_k\rangle \langle x-y, e_k\rangle  \|_{L^2(H, \mu)}}
\\
\\
\ds{\leq \sum_{k=p}^{q}\| \lambda_k^{-\alpha/2}  \langle x , e_k\rangle\|_{L^2(H, \mu)}\| \lambda_k^{-\alpha/2}   \langle x-y, e_k\rangle  \|_{L^2(H, \mu)} }
\\
\\
\ds{ = \sum_{k=p}^{q}  \lambda_k^{(1-\alpha)/2}  \lambda_k^{-\alpha/2} (   \lambda_k  + | \langle  y, e_k\rangle|^2)^{1/2} }
\\
\\
\ds{ \leq \sum_{k=p}^{q}  \lambda_k^{(1-\alpha)/2}(   \lambda_k^{(1-\alpha)/2} + \lambda_k^{-\alpha/2} | \langle  y, e_k\rangle|)}
\\
\\
\ds{
\leq  \sum_{k=p}^{q}  \lambda_k^{1-\alpha } + \frac{1}{2}  (  \lambda_k^{1-\alpha }  +  \lambda_k^{-\alpha} | \langle  y, e_k\rangle| ^2),}
\end{array}$$
where $\sum_{k=1}^{\infty}  \lambda_k^{1-\alpha }  <\infty$  by assumption, and $\sum_{k=1}^{\infty}  \lambda_k^{-\alpha} | \langle  y, e_k\rangle|^2 <\infty $ because $y\in D(A^{\alpha/2})$ . Therefore, 
$$\exists L^2(H, \mu)-\lim_{n\to \infty} \sum_{k=1}^{n}\lambda_k^{-\alpha}  \langle x , e_k\rangle \langle x-y, e_k\rangle :=  \langle x, A^{\alpha}(x-y)\rangle .$$
(Note that $\langle x, A^{\alpha}(x-y)\rangle$ is not defined pointwise).
It follows that $ \rho' (|x_n-y_n|^2)\cdot$ $  \sum_{k=1}^{n}\lambda_k^{-\alpha}  \langle x , e_k\rangle \langle x-y, e_k\rangle$ converges to 
$ \rho' (|(x -y |^2) \langle x, A^{\alpha}(x-y)\rangle$ in $L^2(H, \mu)$. Since $L_{\alpha}$ is closed, $\theta\in D( L_{\alpha})$.

\vspace{3mm}
\noindent {\em Second  step: $ \varphi_{\eps} \theta $ belongs to $D(L_{\alpha})$.}

Since $ \varphi_{\eps} \in D(L_{\alpha})$ and ${\mathcal E}_{\alpha}(H)$ is a core of $L_{\alpha}$, there is a sequence of exponential functions  $\beta_n$ that converges to $ \varphi_{\eps}$ in $D(L_{\alpha})$. Since $\theta$ is bounded,  $\beta_n\theta $ converges to $ \varphi_{\eps} \theta $ in $L^2(H, \mu)$. 
By Lemma \ref{prodotto}, $\beta_n\theta$ belongs to $ D(L_{\alpha})$ for every $n$, and we have 
$$ L_{\alpha}( \beta_n \theta) = \beta_n L_{\alpha}\theta + \theta L_{\alpha}\beta_n + \langle Q^{1-\alpha}D\beta_n, D\theta  \rangle.$$
 As $n \to \infty$, $\beta_n$ converges to $  \varphi_{\eps} $, $L_{\alpha}\beta_n $ converges to $ L_{\alpha} \varphi_{\eps} $, and $\langle Q^{1-\alpha}D\beta_n,   D\theta  \rangle = \langle Q^{(1-\alpha)/2}D\beta_n,  Q^{(1-\alpha)/2} D\theta  \rangle $ converges to $  \langle Q^{(1-\alpha)/2 }D 
 \varphi_{\eps} , Q^{(1-\alpha)/2 }D\theta \rangle $ in $L^2(H, \mu)$ since $D(L_{\alpha})\subset W^{1,2}_{\alpha}(H, \mu)$ and $Q^{(1-\alpha)/2} D\theta$ is bounded. 
Therefore, $ L_{\alpha}( \beta_n \theta)$ converges in $L^2(H, \mu)$, and since $ L_{\alpha}$ is closed, $ \varphi_{\eps} \theta $ belongs to $D(L_{\alpha})$ and 
\begin{equation}
L_{\alpha} ( \theta \varphi_{\eps} ) = (L_{\alpha} \theta ) \varphi_{\eps} +  \langle Q^{(1-\alpha)/2}D\theta, Q^{(1-\alpha)/2}D \varphi_{\eps}\rangle + \theta L_{\alpha}  \varphi_{\eps}.
\label{varphitheta}
\end{equation}

\vspace{3mm}
\noindent {\em Third  step: $\widetilde{\varphi} \theta$ belongs to $W^{2,2}_{\alpha}(H, \mu)$.}
Using \eqref{varphitheta} and  \eqref{e2.3} we get 
 $$\lambda \theta \varphi_{\eps}  - L_{\alpha} (\theta \varphi_{\eps} )  = \theta \widetilde{f }-(L_{\alpha} \theta ) \varphi_{\eps} - \langle Q^{(1-\alpha)/2}D\theta, Q^{(1-\alpha)/2}D \varphi_{\eps}\rangle := f_{1, \eps}.$$
 The $L^2$ norm of the right hand side $f_{1, \eps}$ is bounded by a constant independent of $\eps$. Therefore, 
 $\|\theta \varphi_\eps\|_{D( L_{\alpha})}$ is bounded by a constant independent of $\eps$, and since  $D( L_{\alpha})$ is continuously embedded in $W^{2,2}_{\alpha}(H, \mu)$, also  $\|\theta \varphi_\eps\|_{W^{2,2}_{\alpha}(H, \mu)}$ is .
  
Let $\{\eps_k\}$  be the sequence used in  the proof of Proposition \ref{Identificazione}, so that  $  \varphi_{\eps_k}$ converges weakly in $W^{1,2}_{\alpha}(H, \mu)$ to $\widetilde{\varphi}$. Possibly taking a further subsequence,  
 $(\theta \varphi_{\eps_k})$ converges weakly in $W^{2,2}_{\alpha}(H, \mu)$ to a function $u$ that belongs to $W^{2,2}_{\alpha}(H, \mu)$. 
Then $u = \theta \widetilde{\varphi}$; indeed, for each $\psi \in L^2(H, \mu)$ we have 
$$\int_H u\,\psi\,d\mu =  \lim_{k\to \infty}\int_H \theta \varphi_{\eps_k}\psi\,d\mu =  \lim_{k\to \infty}\int_H \theta \widetilde{\varphi}\psi\,d\mu .$$
So, $\theta \widetilde{\varphi} \in W^{2,2}_{\alpha}(H, \mu)$. \end{proof}


 \section{Domains with smooth boundaries}
 
 
In this section we assume that 
$$K = \{ x\in H:\; g(x) \leq 1\}$$
where $g: H\mapsto \R$ is a $C^1$  function that belongs to $D(L_0)$ and satisfies \eqref{cosaserve}. Moreover we assume that $\sup g >1$, so that $K$ is a proper subset of $H$, and $\inf g <1$, so that the interior part of $K$ is not empty and the surface measure $d\sigma$ is well defined in the boundary $\Sigma$ of $K$,  $\Sigma= \{ x\in H:\; g(x) = 1\}$. See the Appendix, to which we refer for the definition and properties of surface measures.

The aim of this section is to give a reasonable definition of the trace at $\partial K$ of any function in $W^{1,2}_{\alpha}(H, \mu)$, and to show that the functions in $\oo{W}^{1,2}_{\alpha}(H, \mu)$ have null trace at $\partial K$. This   implies  that  $R(\lambda, L^{K}_{\alpha})f$ 
satisfies the Dirichlet boundary condition in \eqref{e1.6a} in the sense of the trace for every $f\in L^2(K, \mu)$, and that $T^{K}_{\alpha}(t)f$ has null trace at the boundary for every $t>0$ and $f \in L^2(K,\mu)$.

As a first step we prove   integration formulas  for functions in the core ${\mathcal E}_{0}(H)$.

\begin{prop}
\label{p5.1}
Let $k\in \N$ be such that $D_kg/|Q^{1/2}Dg|\in W^{2,2}_{0}(H, \mu)$. Then 
for every $\varphi\in  {\mathcal E}_{0}(H)$   we have
\begin{equation}
\label{parti} 
\int_{K} D_k\varphi \,d\mu = \frac{1}{\lambda_k} \int_{K} x_k\varphi \,d\mu + \int_{\Sigma} \frac{D_k g}{|Q^{1/2}Dg| }\varphi \,d\sigma.\end{equation}
If $ |Q^{1/2}Dg|\in W^{2,2}_{0}(H, \mu)$, then 
for every $\varphi\in  {\mathcal E}_{0}(H)$   we have

\begin{equation}
\label{e5.1}
\ds{ \int_{\Sigma}  \varphi^2 |Q^{1/2}Dg|   \,d\sigma_1}   
\left\{ \begin{array}{ll}
=   \ds{ \int_K \varphi \langle Q^{1/2}D\varphi, Q^{1/2}Dg\rangle   \,d\mu +  \int_K L_0g \,\varphi^2    \,d\mu } & (a)
\\
\\
= - \ds{ \int_{K^c} \varphi \langle Q^{1/2}D\varphi, Q^{1/2}Dg\rangle   \,d\mu -  \int_{K^c} L_0g \,\varphi^2    \,d\mu } & (b)
\end{array}\right. 
\end{equation}
\end{prop}
\begin{proof}  For small $\eps >0$ define the pathwise linear function $\theta_{\eps}$ by
$$\theta_{\eps}(\xi) :=\left\{ \begin{array}{ll} 
2, & \xi\leq 1-\eps, 
\\
\frac{1}{ \eps}(1-\xi) +1, & 1-\eps <\xi < 1+\eps, 
\\
0, & \xi \geq 1+\eps.
\end{array}\right., $$
and set 
$$\rho_{\eps}(x) := \theta_{\eps}(g(x)), \quad x\in H.$$
Since $\theta_{\eps}$ is Lipschitz continuous, then $\rho_{\eps}\in W^{1,2}_{0}(H, \mu)$ (\cite[Rem.~5.2.1]{Bo}). Then   the product $\rho_{\eps}\varphi $ belongs to $W^{1,2}_{0}(H, \mu)$ and $D_k (\rho_{\eps}\varphi )= \theta_{\eps}'(g(x))D_k g(x)\varphi(x)  + \rho_{\eps}(x)D_k  \varphi(x)$, so that 
\begin{equation}
\label{prima}
\int_H (D_k\varphi ) \rho_{\eps} \,d\mu   -\frac{1}{ \eps} \int_{1-\eps<g <1+\eps}  \varphi D_k g \,d\mu = \frac{1}{\lambda_k}
\int_H x_k\varphi  \rho_{\eps} \,d\mu, \quad k\in \N.
\end{equation}

Let us prove \eqref{parti}. Letting $\eps\to 0$, $\rho_{\eps}$ converges pointwise  to $2\one_{K}$ in $H\setminus \Sigma$, whose measure is $1$. Since $\rho_{\eps}\leq 2$, by dominated convergence we get  
$$\exists   \lim_{\eps \to 0} \frac{1}{2 \eps} \int_{1-\eps<g <1+\eps}  \varphi D_k g \,d\mu  = \int_{K} D_k\varphi \,d\mu - \frac{1}{\lambda_k} \int_{K} x_k\varphi \,d\mu .$$
Let us identify the limit in the left hand side as a boundary integral. Since $\varphi D_kg |Q^{1/2}Dg|^{-1} $ $\in $ $W^{2,2}_{0}(H, \mu)$, by Remark \ref{rem:limite} we have
$$\lim_{\eps \to 0}  \frac{1}{2\eps} \int_{1-\eps<g <1+\eps}    \varphi D_k g \,d\mu  = \int_{\Sigma } \frac{D_kg}{ |Q^{1/2}Dg|}  \varphi   \,d\sigma $$
and \eqref{parti} follows. 

Let us prove   \eqref{e5.1}(a). For every $\eps >0$ and $k\in \N$, the function $\rho_{\eps}\varphi ^2D_kg$ still belongs to $W^{1,2}_{0}(H, \mu)$. Therefore we may replace  $\varphi $ in \eqref{prima}  by $\lambda_k \varphi^2D_kg$,  and 
summing over $k$ (recall Lemma   \ref{emb}), we obtain 
$$\int_H 2\varphi \langle Q^{1/2}D\varphi, Q^{1/2}Dg\rangle  \rho_{\eps} \,d\mu + \int_H 2L_0g \,\varphi^2  \rho_{\eps} \,d\mu $$
$$= \frac{1}{ \eps}  \int_{1-\eps<g <1+\eps}  \varphi^2 |Q^{1/2}Dg|^2  \,d\mu$$
Letting $\eps\to 0$ as before, by dominated convergence we get  
$$\lim_{\eps \to 0}  \int_H \varphi  \langle Q^{1/2}D\varphi, Q^{1/2}Dg\rangle  \rho_{\eps} \,d\mu  = 
  \int_K \varphi \langle Q^{1/2}D\varphi, Q^{1/2}Dg\rangle   \,d\mu, $$
$$\lim_{\eps \to 0}  \int_H L_0g \,\varphi^2  \rho_{\eps} \,d\mu =    \int_K L_0g \,\varphi^2    \,d\mu .$$
Therefore, there exists the limit 
$$\lim_{\eps \to 0}  \frac{1}{2\eps} \int_{1-\eps<g <1+\eps}  \varphi^2 |Q^{1/2}Dg|^2  \,d\mu =  \int_K \varphi \langle Q^{1/2}D\varphi, Q^{1/2}Dg\rangle   \,d\mu +  \int_K L_0g \,\varphi^2    \,d\mu $$
that we identify as a boundary integral. Indeed, since $\varphi^2  |Q^{1/2}Dg| \in W^{2,2}_{0}(H, \mu)$,  by  Remark \ref{rem:limite}  we have 
$$\lim_{\eps \to 0}  \frac{1}{2\eps} \int_{1-\eps<g <1+\eps}  \varphi^2 |Q^{1/2}Dg|^2  \,d\mu = \int_{\Sigma}  \varphi^2 |Q^{1/2}Dg|   \,d\sigma.$$
So,  \eqref{e5.1}(a) holds. To prove \eqref{e5.1}(b), we may follow the same procedure replacing $K$ by $K^c$ and $\theta_{\eps}$ by
$$\widetilde{\theta}_{\eps}(\xi) :=\left\{ \begin{array}{ll} 
0, & \xi\leq 1-\eps, 
\\
\frac{1}{ \eps}( \xi -1) +1, & 1-\eps <\xi < 1+\eps, 
\\
2, & \xi \geq 1+\eps ,
\end{array}\right. $$
or else, we may use the equality 
$$  \int_{K^c} \varphi \langle Q^{1/2}D\varphi, Q^{1/2}Dg\rangle   \,d\mu +  \int_{K^c} L_0g \,\varphi^2    \,d\mu $$
$$ = -  \int_K \varphi \langle Q^{1/2}D\varphi, Q^{1/2}Dg\rangle   \,d\mu - \int_K L_0g \,\varphi^2    \,d\mu  $$
that follows from 
$$ \int_H L_0g \, \varphi^2 d\mu = - \frac{1}{2}\int_H \langle Q^{1/2}Dg, Q^{1/2}D(\varphi^2) \rangle d\mu 
= - \int_H \langle Q^{1/2}Dg, Q^{1/2}D \varphi ) \rangle \varphi \, d\mu $$
(see formula \eqref{e1.6aa}). 
\end{proof}

As a second step, with the aid of Proposition \ref{p5.1} we   prove an integration by parts formula in $W^{1,2}_{0}(H, \mu)$ and 
 we   define the {\em trace} $\varphi_{|\Sigma}$ at the boundary $\Sigma$ of any function in $W^{1,2}_{0}(H, \mu)$.

\begin{cor}
\label{MaggTraccia}
Assume that $|Q^{1/2}Dg|\in W^{2,2}_{0}(H, \mu)$, and that $|Q^{1/2}Dg|$ is bounded and $L_0g$ has at most linear growth either on $K$ or on $K^c$. Then for every   $\varphi \in W^{1,2}_{0}(H, \mu)$ there exists $\psi \in L^2(\Sigma, \sigma)$ with the following property: for each sequence $(\varphi_n)\in  {\mathcal E}_{0}(H)$ such that $\lim_{n\to \infty} \|\varphi_n - \varphi \|_{W^{1,2}_{0}(H, \mu)} =0$, the sequence $(\varphi_n |Q^{1/2} Dg|^{1/2}_{|\Sigma})$ converges to $\psi$  in $L^2(\Sigma, \sigma)$.
\end{cor}
\begin{proof}  It is sufficient to   recall formula  \eqref{e5.1} and Lemma \ref{emb}. 
\end{proof}

Note that the assumptions of Corollary \ref{MaggTraccia} are satisfied by the functions $g$ in Example \ref{examples} of the Appendix.

\begin{defn} Under the assumptions of Corollary 
\ref{MaggTraccia}, for each $\varphi \in W^{1,2}_{0}(H, \mu)$ the trace of $\varphi$ at $\Sigma$  is defined by 
$$\varphi_{| \Sigma} = \frac{\psi}{|Q^{1/2} Dg|^{1/2}},$$
where $\psi$ is given by Corollary 
\ref{MaggTraccia}. 
\end{defn}

Note that in general $\varphi_{|\Sigma} $ does not belong to $L^2(\Sigma, \sigma)$, 
because $ |Q^{1/2} Dg|^{-1/2}$ may be unbounded in $\Sigma$. Of course, if $ |Q^{1/2} Dg|^{-1/2}$ is   bounded in $\Sigma$ (that is, if $\inf_{\Sigma} |Q^{1/2} Dg| >0$), then $\varphi_{|\Sigma} \in L^2(\Sigma, \sigma)$ for every $\varphi \in W^{1,2}_{0}(H, \mu)$ and 
the mapping $W^{1,2}_{0}(H, \mu)\mapsto L^2(\Sigma, \sigma)$, 
$\varphi \mapsto \varphi_{|\Sigma}$ is continuous. 

In general, we have the following lemma.

\begin{lem}
\label{tracciaL1}
Under the assumptions of Corollary \ref{MaggTraccia}, for every $\varphi \in W^{1,2}_{0}(H, \mu)$, $\varphi_{|\Sigma} \in L^1(\Sigma, \sigma)$ and the mapping $W^{1,2}_{0}(H, \mu)$ $\mapsto $ $L^1(\Sigma, \sigma)$, 
$\varphi \mapsto \varphi_{|\Sigma}$ is continuous. 
\end{lem}
\begin{proof}  Since $\varphi_{|\Sigma} = \psi |Q^{1/2} Dg|^{-1/2}$ with $\psi \in L^2(\Sigma, \sigma)$, it is sufficient to prove that 
  $|Q^{1/2} Dg|^{-1/2} \in L^2(\Sigma, \sigma)$.  The assumptions $  \|Q^{1/2}D^2g\,Q^{1/2}\|_{{\mathcal L}(H)}/ |Q^{1/2}Dg |^2\in L^{2}(H, \mu)$ and $|Q^{1/2} Dg|^{-1 }\in L^4(H, \mu)$, that are contained in assumption \eqref{cosaserve}, imply that the function $\widetilde{\varphi} :=  |Q^{1/2} Dg|^{-1 }$ belongs to 
$W^{1,2}_{0}(H, \mu)$. By   Corollary \ref{MaggTraccia}, $\widetilde{\varphi} |Q^{1/2} Dg|^{1/2} $ $=$ $ |Q^{1/2} Dg|^{-1/2}$ has trace in $L^2(\Sigma, \sigma)$.
\end{proof}

\begin{cor}
Let the assumptions of Corollary \ref{MaggTraccia} be satisfied. The following statements hold for every $\alpha \in [0,1]$. 
\label{PartiTraccia}
\begin{itemize}
\item[(i)] If $D_kg/|Q^{1/2}Dg| \in W^{2,2}_{0}(H, \mu)$, for every $\varphi \in W^{1,2}_{\alpha}(H, \mu)$  the integration by parts 
formula \eqref{parti}  holds. 
\item[(ii)] If $\varphi \in \oo{W}^{1,2}_{\alpha}(K, \mu)$, its trace at $\Sigma_1$ vanishes. 
\end{itemize}
\end{cor}
\begin{proof} Since $ W^{1,2}_{\alpha}(H, \mu)\subset  W^{1,2}_{0}(H, \mu)$, and $ \oo{W}^{1,2}_{\alpha}(K, \mu)\subset  \oo{W}^{1,2}_{0}(K, \mu)$, it is enough to prove that the statements hold for $\alpha =0$. 

\noindent  (i) It is sufficient to approach   every $\varphi \in W^{1,2}_{0}(H, \mu)$ by a sequence $(\varphi_n)\subset  {\mathcal E}_{0}(H)$, and to recall  Lemma \ref{tracciaL1}. 

\noindent (ii) If $\varphi \in \oo{W}^{1,2}_{0}(K, \mu)$, it vanishes a.e. in $K^c$, and   formula \eqref{e5.1}(b) yields the statement.
\end{proof}



\appendix

 \section{Surface integrals}

We consider level surfaces of smooth functions  $g$. We refer to \cite[\S 6.10]{Bo}, where the functions $g$ under consideration belong to the space $W^{\infty}(H, \mu)$ defined by 
$$W^{\infty}(H, \mu):=  \bigcap _{k\in \N, p>1}W^{k,p}(H, \mu)$$
and $W^{k,p}(H, \mu)$ is the completion of the smooth cylindrical   functions$^($\footnote{that is,  functions of the type  $f(x) =\varphi( \langle x, x_1\rangle, \ldots , \langle x, x_n\rangle)$ with $x_1$, \ldots, $x_n \in H$ and $\varphi \in C^{\infty}_{b}(\R^n)$.}$^)$ in the norm
$$\|f\|_{k,p} : = \|f\|_{L^p(H, \mu) } +  \sum_{j=1}^{k} \bigg( \int_H \bigg[ \sum_{i_1, \ldots, i_j\geq 1}  (\lambda_{i_1}\cdot \ldots \cdot \lambda_{i_k}D_{i_1}\ldots D_{i_k}f(x))^2  \bigg]^{p/2} \mu(dx) \bigg)^{1/p}$$
(In particular, the spaces $W^{k,2}(H, \mu)$ coincide  with our $W^{k,2}_{0}(H, \mu)$ for $k=1, 2$). 

Another assumption is 
$$|Q^{1/2}Dg|^{-1}\in  \bigcap _{  p>1}L^{p}(H, \mu).$$

Our aim here is to give a simplified presentation of surface measures in the case of a Hilbert space setting, under less heavy (although less elegant) assumptions on $g$.

For any continuous $g:H\mapsto \R$ and $r$ in the range of $g$ let us define  the level sets 
$$\Sigma_r:= \{x\in H:\; g(x)= r \}. $$
We shall define  probability measures on the surfaces $\Sigma_r$ with $r$ in the interior part of the range of $g$. To this aim, a first step is the study of  the image of $\mu$ on $\R$ under the mapping $g$, defined by 
 $$(\mu \circ g^{-1})(I) := \mu(g^{-1}(I)), \quad I\in \B(\R).$$
We shall give sufficient conditions for  $\mu \circ g^{-1}$ have continuous (in fact, $W^{1,2}$) density $k$ with respect to the Lebesgue measure. Similarly, for $\rho \in L^1(H, \mu)$ we shall consider the signed measure   
$$(\rho \mu)(B) := \int_B \rho(x)\mu(dx), \quad B\in \B(H)$$
and its image under the mapping $g$,
 $$(\rho \mu \circ g^{-1})(I) := (\rho \mu)(g^{-1}(I)), \quad I\in \B(\R),$$
and we shall give sufficient conditions for $\rho \mu \circ g^{-1}$ have continuous density $k_{\rho}$  with respect to the Lebesgue measure. 
A key role will be played by the function $\psi$ defined by 
\begin{equation}
\label{psi}
\psi :=     \frac{L_0g }{ |Q^{1/2}Dg|^2}  - \frac{\langle Q^{1/2}D^2g \,Q^{1/2} \cdot Q^{1/2}Dg, Q^{1/2}Dg\rangle}{ |Q^{1/2}Dg|^4} ,
\end{equation}
if $g\in D(L_0)$. 
We shall use the following lemma.

\begin{prop}
\label{Pr:ek}
Let  $g\in D( L_0)$ be such that $|Q^{1/2}Dg|^{-1}\in L^4(H, \mu)$. Then 
\begin{itemize}
\item[(a)] 
$\mu \circ g^{-1}$ is absolutely continuous with respect to the Lebesgue measure. 
\item[(b)] 
If a function  $\rho \in  W^{1,1}_{0}(H, \mu)$ is such that 
\begin{equation}
\label{rho}
\psi \rho   \in L^1(H, \mu), \quad \frac{|Q^{1/2}D\rho|}{|Q^{1/2}Dg|} \in L^1(H, \mu), 
\end{equation}
where $\psi$ is defined in \eqref{psi},   then   $\rho \mu \circ g^{-1}$ is absolutely continuous with respect to the Lebesgue measure. 
\end{itemize}
\end{prop}
\begin{proof}
To prove  statement (a) we shall show that there exists  $C>0$ such that 
\begin{equation}
\label{condsuff}\bigg| \int_{\R}\varphi'(r) (\mu \circ g^{-1})(dr) \bigg| \leq C\| \varphi\|_{\infty}, \quad \varphi \in C_b^1(\R).
\end{equation}

For each  $k\in \N$ we have 
\begin{equation}
\label{dercomp} D_k( \varphi \circ g )(x) = \varphi'(g(x))D_kg(x) , \quad x\in H, 
\end{equation}
so that  
\begin{equation}
\label{gradcomp} \langle D ( \varphi \circ g) (x) , QDg(x)\rangle = ( \varphi' \circ g)(x)|Q^{1/2}Dg(x) |^2, \quad x\in H, 
\end{equation}
i.e.
\begin{equation}
\label{comp}(  \varphi' \circ g)(x) = \frac{  \langle Q^{1/2}D ( \varphi \circ g) (x) , Q^{1/2}Dg(x)\rangle }{|Q^{1/2}Dg(x) |^2}, \quad a.e.\; x\in H.  
\end{equation}
Therefore, 
$$  \int_{\R}\varphi'(r)( \mu \circ g^{-1})(dr) = \int_H  \varphi' \circ g\,d\mu = \int_H  \frac{ \sum_{k}\lambda_k   D_k ( \varphi \circ g) (x)   D_kg(x)  }{|Q^{1/2}Dg(x) |^2} d\mu .$$
Integrating by parts and recalling that
\begin{equation}
\label{derquoz} D_k \bigg(  \frac{1}{ |Q^{1/2}Dg|^2} \bigg) = -2 \;\frac{\sum_{i} \lambda_iD_ig D_{ik}g}{ |Q^{1/2}Dg|^4} 
\end{equation}
we obtain 
$$\begin{array}{ll}
\ds {\int_H  \varphi' \circ g\,d\mu } & =    \ds{- \int_H  \varphi \circ g \sum_{k} \lambda_k D_k\bigg( \frac{D_kg}{ |Q^{1/2}Dg|^2}\bigg) d\mu 
+\int_H \varphi \circ g \sum_{k}  \frac{x_kD_kg}{ |Q^{1/2}Dg|^2}\,d\mu}
\\
\\
& =  \ds{- \int_H  \varphi \circ g \sum_{k} \lambda_k  \bigg( \frac{D_{kk}g }{ |Q^{1/2}Dg|^2} -2D_kg   \frac{\sum_{i} \lambda_iD_ig D_{ik}g}{ |Q^{1/2}Dg|^4} \bigg)d\mu}
\\
\\
&  +  \ds{ \int_H \varphi \circ g \sum_{k}  \frac{x_kD_kg}{ |Q^{1/2}Dg|^2}\,d\mu}
\\
\\
& =   \ds{ -  2 \int_H   (\varphi \circ g)(x)\psi(x)d\mu},
\end{array}$$
where the function $\psi$ is defined in  \eqref{psi}.  The first addendum in $\psi$, $ L_0g / |Q^{1/2}Dg|^2$, belongs to 
 $L^1(H, \mu)$ since both $ L_0g $ and  $1 / |Q^{1/2}Dg|^2$ are in $L^2(H, \mu)$.  Concerning the second addendum we have
$$   \frac{|\langle Q^{1/2}D^2g\,Q^{1/2}\cdot Q^{1/2}Dg, Q^{1/2}Dg\rangle |}{ |Q^{1/2}Dg|^4}   \leq \frac{ \| Q^{1/2}D^2g\,Q^{1/2}\|_{{\mathcal L}(H)}}{ |Q^{1/2}Dg|^2} .$$
Recalling that  there exists $C_0>0$ such that (\cite[Thm.~ 5.7.1]{Bo}) 
$$\|x\mapsto  \| Q^{1/2}D^2g\,Q^{1/2}\|_{{\mathcal L}(H)}\|_{L^2(H, \mu)} \leq C_0 \|g\|_{D(L_0)},$$
it follows that the second addendum in $\psi$ belongs to 
 $L^1(H, \mu)$. Then formula \eqref{condsuff}   follows, with  $C= \|\psi\|_{L^1(H, \mu)} \leq $ const. $( \|g\|_{D(L_0)} + \| |Q^{1/2}Dg|^{-1}\|_{L^4(H. \mu)})$. 
 
\vspace{3mm}

We prove  statement (b) by the same procedure, replacing $\mu$ by $\rho \mu$. For every  $\varphi \in C_b^1(\R)$ we have 
$$\begin{array}{lll}
\ds {\int_H ( \varphi' \circ g)\,\rho \, d\mu } & = &\ds{  \int_H   \sum_{k}\lambda_k   D_k ( \varphi \circ g) (x)   D_kg(x)\frac{\rho(x)  }{|Q^{1/2}Dg(x) |^2} d\mu }
\\
\\
&=&  \ds{  \int_H  \varphi \circ g \bigg( -2 \psi \rho - \frac{\langle Q^{1/2}Dg, Q^{1/2}D\rho\rangle}{|Q^{1/2}Dg(x) |^2} \bigg) d\mu }
\end{array}$$
where  $\psi$ is the function defined  in \eqref{psi}.  Assumption  \eqref{rho} implies that the functions $\psi \rho$ and $ \langle Q^{1/2}Dg, Q^{1/2}D\rho\rangle/|Q^{1/2}Dg(x) |^2$  belong  to  $L^1(H, \mu)$. Then, 
$$\bigg| \int_{\R}\varphi'(r) (\mu \circ g^{-1})(dr) \bigg|  = \bigg| \int_H ( \varphi' \circ g)\,\rho \, d\mu \bigg| 
\leq C\| \varphi\|_{\infty}, \quad \varphi \in C_b^1(\R)$$
with   $C = 2\|\psi\rho\|_{L^1} + \| \frac{|Q^{1/2}D\rho|}{|Q^{1/2}Dg|}\|_{L^1}$.  The statement follows. 
\end{proof}

\begin{prop}
\label{Pr:contk}
Let the assumptions of Proposition \ref{Pr:ek} hold. Then:
\begin{itemize}
\item[(a)] If the function $\psi$ defined in \eqref{psi} belongs to  $W^{1,2}_{0}(H, \mu)$, then the density $k$ of $\mu \circ g^{-1}$ belongs to  $W^{1,1}(\R)$. 
\item[(b)] If $\rho \in  W^{1,1}_{0}(H, \mu)$ satisfies \eqref{rho} and moreover, setting 
$$\rho_1 :=  2\psi \,\rho + \frac{\langle Q^{1/2}Dg, Q^{1/2}D\rho\rangle}{ |Q^{1/2}Dg |^2}$$
we have  $\rho_1 \in  W^{1,1}_{0}(H, \mu)$,  $\psi \rho_1   \in L^1(H, \mu)$, $ \frac{|Q^{1/2}D\rho_1|}{|Q^{1/2}Dg|} \in L^1(H, \mu)$, then 
$k_{\rho}\in W^{1,1}(\R)$. 
\end{itemize}
\end{prop}
\begin{proof}
To prove statement (a) we 
shall show that there is $C_1>0$ such that 
$$\bigg| \int_{\R}\varphi ''(r) (\mu \circ g^{-1})(dr) \bigg| \leq C_1\| \varphi\|_{\infty}, \quad \varphi \in C_b^2(\R).$$
Indeed, this  implies  that   $k$ is weakly differentiable with $k ' \in L^1(\R)$. 

Differentiating  \eqref{dercomp} we get
$$D_{kk}( \varphi \circ g )(x) = \varphi''(g(x))(D_kg(x))^2 +   \varphi'(g(x))D_{kk}g(x), \quad x\in H, $$
and summing over  $k$
$${\rm Tr}(QD^2(g\circ \varphi)) = \varphi''(g(x)) |Q^{1/2}Dg(x) |^2 +  \varphi'(g(x)) {\rm Tr}(QD^2g(x))$$
so that 
$$\begin{array}{lll}
 \varphi'' \circ g & = &\ds{ \frac{ {\rm Tr}(QD^2( \varphi \circ g)) }{ |Q^{1/2}Dg |^2} - (\varphi'\circ g) \frac{ {\rm Tr}(QD^2g) }{ |Q^{1/2}Dg  |^2} }
 \\
 \\
& =  & \ds{ \frac{ 2L_0( \varphi \circ g) +\langle x, D( \varphi \circ g)\rangle }{ |Q^{1/2}Dg |^2}  - (\varphi'\circ g)  \frac{ 2L_0g + \langle x, Dg \rangle}{ |Q^{1/2}Dg |^2} }
\\
\\
& = & \ds{ \frac{ 2L_0( \varphi \circ g) }{ |Q^{1/2}Dg |^2}  - (\varphi'\circ g)  \frac{ 2L_0g }{ |Q^{1/2}Dg |^2} .}
\end{array}$$
Using again  \eqref{derquoz} we get
$$\begin{array}{l}
\ds {\int_H  ( \varphi'' \circ g) d\mu }  =
\\
\\
=  \ds{  \int_H \bigg( -\langle Q^{1/2}D( \varphi \circ g) , Q^{1/2}D( |Q^{1/2}Dg |^{-2})\rangle -2 (\varphi'\circ g)  \frac{L_0g}{ |Q^{1/2}Dg |^2} \bigg) d\mu }
\\
\\
=  \ds{   \int_H \varphi'\circ g \bigg( \langle Q^{1/2}Dg, 2 \frac{Q^{1/2}D^2g \,Q^{1/2}\cdot  Q^{1/2}Dg}{ |Q^{1/2}Dg |^4}\rangle - 2  \frac{L_0g}{ |Q^{1/2}Dg |^2} \bigg) d\mu }
 \\
 \\
=  \ds{  -2 \int_H (\varphi'\circ g ) \psi\,d\mu} ,
\end{array}$$
where  $\psi$ is defined in  \eqref{psi}.  Then we may use Proposition \ref{Pr:ek}, with 
$\rho = \psi$. By assumption,   $\psi\in W^{1,2}_{0}(H, \mu)\subset W^{1,1}_{0}(H, \mu)$, moreover $\psi^2 \in L^1(H, \mu)$ and  $ \frac{|Q^{1/2}D\psi|}{|Q^{1/2}Dg|}\in L^1(H, \mu)$ since  $|Q^{1/2}D\psi|\in L^2(H,  \mu)$,  $|Q^{1/2}Dg|^{-1}\in L^2(H. \mu)$. We get 
 $|  \int_H (\varphi'\circ g ) \psi\,d\mu| \leq C \|\psi\|_{W^{1,2}_{0}(H, \mu)} \|\varphi\|_{\infty}$, and statement (a) follows.

\vspace{3mm}

 Concerning statement (b), the proof is similar, replacing 
  $\mu$ by $\rho \mu$. For every  $\varphi \in C_b^2(\R)$ we have
$$\begin{array}{l}
\ds {\int_H  ( \varphi'' \circ g) \rho d\mu  =  \int_H  \bigg( \frac{ 2\rho L_0( \varphi \circ g) }{ |Q^{1/2}Dg |^2}  - 2(\varphi'\circ g)  \frac{ \rho L_0g }{ |Q^{1/2}Dg |^2}\bigg) d\mu }
\\
\\
= \ds{  \int_H \bigg( -\langle Q^{1/2}D( \varphi \circ g) , Q^{1/2}D\bigg( \frac{\rho}{ |Q^{1/2}Dg |^{2} }\bigg)\rangle -2 (\varphi'\circ g)  \frac{\rho L_0g}{ |Q^{1/2}Dg |^2} \bigg) d\mu }
 \end{array}$$
$$\begin{array}{l}
=   \ds{   \int_H \varphi'\circ g   \bigg( \langle Q^{1/2}Dg, 2 \frac{Q^{1/2}D^2g \,Q^{1/2}\cdot  Q^{1/2}Dg}{ |Q^{1/2}Dg |^4}\rangle 
- 2  \frac{L_0g}{ |Q^{1/2}Dg |^2} \bigg) \rho \,d\mu }
\\
\\
 - \ds{   \int_H \varphi'\circ g  \frac{\langle Q^{1/2}Dg, Q^{1/2}D\rho\rangle}{ |Q^{1/2}Dg |^2}\,  d\mu }
 \\
 \\
=  - \ds{   \int_H  \varphi'\circ g \bigg( 2 \psi \,\rho + \frac{\langle Q^{1/2}Dg, Q^{1/2}D\rho\rangle}{ |Q^{1/2}Dg |^2} \bigg)\,d\mu =  -  \int_H  \varphi'\circ g\,\rho_1 \,d\mu ,}
\end{array}$$
where the function   $\rho_1 $ satisfies the assumptions of Proposition  \ref{Pr:ek}(b).  We obtain 
$| \int_H ( \varphi'\circ g )\rho_1d\mu|\leq C\|\varphi \|_{\infty}$ and the statement follows.
\end{proof}

One can play with $\rho$ and $g$ in order that the assumptions 
of Proposition  \ref{Pr:contk}(b) are satisfied. In the next proposition we give sufficient conditions that are useful for the sequel.

\begin{prop}
\label{Rem1}
The assumptions of Proposition  \ref{Pr:contk}(b) are satisfied by every   $\rho \in W^{2,2}_{0}(H, \mu)$ provided $g\in D( L_0)$ is  such that 
\begin{equation}
\label{cosaserve}
\left\{ \begin{array}{l}|Q^{1/2}Dg|^{-1}\in L^4(H, \mu), \quad \psi\in W^{1,4}_{0 }(H, \mu) , 
\\
\\ \frac{ \|Q^{1/2}D^2g\,Q^{1/2}\|_{{\mathcal L}(H)}}{|Q^{1/2}Dg |^2}\in L^{2}(H, \mu), \;  \frac{ \|Q^{1/2}D^2g\,Q^{1/2}\|_{{\mathcal L}(H)}}{|Q^{1/2}Dg |^3}\in L^2(H, \mu). 
\end{array}\right.
\end{equation}
In this case there exists  $C_2>0$, depending only on $g$, such that  
$$\bigg| \int_{\R}\varphi''(r) (\rho \mu \circ g^{-1})(dr) \bigg| \leq C_2\|\rho\|_{W^{2,2}_{0}(H, \mu)}
\| \varphi\|_{\infty}, \quad \varphi \in C_b^2(\R).$$
Consequently, if  $\rho_n \to \rho$ in $W^{2,2}_{0}(H, \mu)$ then  $k_{\rho_n}\to k_{\rho}$ in $W^{1,1}(\R)$, hence  $k_{\rho_n}\to k_{\rho}$ in $L^{\infty}(\R)$. 
\end{prop} 
\begin{proof} Since  $\psi\in L^2$ and  $|Q^{1/2}Dg |^{-1}\in L^4$, then $\rho_1\in L^1$. Computing  $Q^{1/2}D\rho_1$ we obtain
$$
\begin{array}{l}
Q^{1/2}D\rho_1=  
\\
\\
= \rho Q^{1/2}D\psi  + \psi Q^{1/2}D\rho - \displaystyle{ \frac{Q^{1/2}D^2g\,Q^{1/2}\cdot Q^{1/2}D\rho + Q^{1/2}D^2\rho Q^{1/2} \cdot Q^{1/2}Dg}{|Q^{1/2}Dg|^2} }
\\
\\
+2 \langle Q^{1/2}Dg, Q^{1/2}D\rho \rangle \displaystyle{ \frac{Q^{1/2}D^2g\,Q^{1/2}\cdot Q^{1/2}Dg}{|Q^{1/2}Dg|^4}}. 
\end{array}$$
Estimating each addendum we get
\begin{itemize}
\item 
$\rho |Q^{1/2}D\psi| \in L^1$, since $ |Q^{1/2}D\psi| \in L^2$; 
\item
$\psi | |Q^{1/2}D\rho| \in L^1$, since $\psi\in  L^2$; 
\item 
$\displaystyle{\frac{\|Q^{1/2}D^2g\,Q^{1/2}\|_{{\mathcal L}(H)}| Q^{1/2}D\rho | }{|Q^{1/2}Dg|^2}  }\in L^1$, since $ \displaystyle{\frac{\|Q^{1/2}D^2g\,Q^{1/2}\|_{{\mathcal L}(H)} }{|Q^{1/2}Dg|^2} }\in L^2$; 
\item
$\displaystyle{ \frac{\|Q^{1/2}D^2\rho Q^{1/2}\|_{{\mathcal L}(H)}}{|Q^{1/2}Dg|}  }\in L^1$, since  $\displaystyle{ \frac{1 }{|Q^{1/2}Dg|} }\in L^2$; 
\item 
$|Q^{1/2}D\rho|  \displaystyle{\frac{\|Q^{1/2}D^2g\,Q^{1/2}\|_{{\mathcal L}(H)} }{|Q^{1/2}Dg|^2} }\in L^1$, as above.
\end{itemize}
Therefore  $\rho_1\in L^1$, and $\|\rho_1\|_{L^1(H, \mu)}  \leq  c\|\rho\|_{W^{2,2}_{0}(H, \mu)}$. 

The assumptions $\psi \in L^4$,  $ \frac{1}{|Q^{1/2}Dg |}\in L^4$ imply that $\psi \rho_1\in L^1$. 

To check that $ \frac{|Q^{1/2}D\rho_1|}{|Q^{1/2}Dg|} \in L^1$ we  redo the estimates above, dividing each term by  $|Q^{1/2}Dg|$.  We get
\begin{itemize}
\item 
$\rho \displaystyle{\frac{ |Q^{1/2}D\psi| }{|Q^{1/2}Dg|}}
\in L^1$, since  $ |Q^{1/2}D\psi| \in L^4$ and  $\displaystyle{ \frac{1 }{|Q^{1/2}Dg|}} \in L^4$; 
\item
$\psi \displaystyle{ \frac{ | |Q^{1/2}D\rho| }{|Q^{1/2}Dg|}}
\in L^1$, since  $\psi\in  L^4$ and  $\displaystyle{ \frac{1 }{|Q^{1/2}Dg|}} \in L^4$; 
\item 
$\displaystyle{\frac{\|Q^{1/2}D^2g\,Q^{1/2}\|_{{\mathcal L}(H)}| Q^{1/2}D\rho | }{|Q^{1/2}Dg|^3}} \in L^1$, since  $\displaystyle{ \frac{\|Q^{1/2}D^2g\,Q^{1/2}\|_{{\mathcal L}(H)} }{|Q^{1/2}Dg|^3}} \in L^2$; 
\item
$\displaystyle{ \frac{\|Q^{1/2}D^2\rho Q^{1/2}_{{\mathcal L}(H)}}{|Q^{1/2}Dg|^2} } \in L^1$, since $\displaystyle{ \frac{1 }{|Q^{1/2}Dg|}} \in L^4$; 
\item 
$|Q^{1/2}D\rho| \displaystyle{\frac{\|Q^{1/2}D^2g\,Q^{1/2}\|_{{\mathcal L}(H)} }{|Q^{1/2}Dg|^3}} \in L^1$, as above.
\end{itemize}
Therefore, the norms    $\|\psi \rho_1\|_{L^1}$ and  $\|  \frac{|Q^{1/2}D\rho_1|}{|Q^{1/2}Dg|}\|_{L^1}$ are bounded by  $c\|\rho\|_{W^{2,2}_{0}(H, \mu)}$. 
Applying Proposition   \ref{Pr:contk}(b)  the statement follows.     \end{proof}

\begin{ex}
\label{examples}
Let us consider some simple examples. 
\begin{itemize}
\item[(a)] $g(x) = \langle b, x\rangle$,  with $|b| =1$, 
\item[(b)] $g(x) =\langle Tx, x\rangle$, with $T\in {\mathcal L}(H)$, $Te_k = t_ke_k$ for each $k\in \N$ and $t_k\neq 0$ for infinitely many $k$, 
\item[(c)] $g(x) = \sum_{k=1}^{13}x_k^2$.
\end{itemize}

In all these cases $g$ satisfies the conditions of Proposition \ref{Rem1}.  
\end{ex}
\begin{proof} In case (a) we have $Dg = b$, $D^2g =0$ so that  $L_0 g = - \langle b, x\rangle /2= -g/2$ and 
$$\psi = -\frac{   \langle b, x\rangle }{2|Q^{1/2}b|^2} $$
which belongs to  $W^{1,4}_{0} (H, \mu)$. The other conditions of Proposition \ref{Rem1} are obviously satisfied. 

\vspace{1mm}

In case (b) we have  $Dg(x)  =2Tx$, $D^2g(x) = 2T$ so that $L_0 g =  {\rm Tr}[QT] - g$ and 
\begin{equation}
\label{psies}
\psi(x) =  \frac {{\rm Tr}[QT] - \langle Tx, x\rangle}{2|Q^{1/2}Tx|^2 } -  \frac {\langle Q^2T^3x , x\rangle }{|Q^{1/2}Tx|^2 } . 
\end{equation}
Since $t_k\neq 0$ for infinitely many $k$, then   $x\mapsto |Q^{1/2}Dg(x)|^{-1}$  belongs to all spaces  $L^p(H, \mu)$.  Indeed, 
$ |Q^{1/2}Dg(x)|^2 \geq 4 \sum_{k=1}^{N} \lambda_i t_k^2x_k^2 $ where $N$ is so large that at least $[p]+1$ addenda do not vanish.
The other assumptions 
of Remark \ref{Rem1} are easily seen to be satisfied. 

\vspace{1mm}

In case (c) we   still have $g(x) = \langle Tx, x\rangle$ with $T\in   {\mathcal L}(H)$, $Tx = \sum_{k=1}^{13} x_k e_k $, so that    $t_k\neq 0$ only for $k=1, \ldots, 13$. However, $ |Q^{1/2}Dg(x)|^{-1} \leq c_0( \sum_{k=1}^{13} x_k^2)^{-1/2}$ with $c_0= 1/\min\{ \lambda_{k}^{1/2}:\; k=1, \ldots, 13\}$ so that $ |Q^{1/2}Dg |^{-1} \in L^p(H, \mu)$ for every $p<13$. The function $\psi$ is still given by \eqref{psies} on   span$\{e_1, \ldots, e_{13}\}$ and it belongs to $ L^p(H, \mu)$ for every $p<13/3$, in particular it belongs to $ L^4(H, \mu)$, as well as $|Q^{1/2}D\psi|^{-1}$. The other conditions 
of Proposition \ref{Rem1} are easily seen to be satisfied.
\end{proof}

In cases (a) and (b) with $T=I$ it is possible to give a representation formula for $k$ that shows   that $k\in C^{\infty}$, see \cite{Hertle}.  In case (c) 
we have $ |Q^{1/2}Dg(x)|^{-1} \geq c_1 ( \sum_{k=1}^{13} x_k^2)^{-1/2}$ with $c_1= 1/\max\{\lambda_{k}^{1/2}:\; k=1, \ldots, 13\}$ so that $ |Q^{1/2}Dg |^{-1} \notin L^p(H, \mu)$ for   $p\geq 13$.

\vspace{3mm}

The construction of the surface measures  goes as follows. First, one constructs surface measures depending explicitly on $g$ by an approximation procedure. 

One fixes once and for all a convex  compact  set $K$ which is symmetric with respect to the origin and has positive measure, say $\mu(K)>1/2$. 
Such a $K$ does exist. Indeed, it is well known that 
 there are compact sets $\widetilde{K}$ with positive (arbitrarily close to $1$) measure (a simple proof is e.g. in \cite[Thm. 6.2]{DP}). The absolute convex hull $K$ of $\widetilde{K}$ is compact, symmetric with respect to the origin and contains $\widetilde{K}$, so that $\mu(K)\geq \mu(\widetilde{K})$.

Then we need a regular cutoff function. The proof of its existence follows closely \cite[Prop. 5.4.12]{Bo}, with a few simplifications 
due to our Hilbert space setting.

\begin{lem}
Let  $K\subset H$ be compact, convex, symmetric with respect to the origin, with $\mu(K) >1/2$. Then there exists a function  $\theta\in W^{\infty}(H, \mu)$ such that  $\theta \equiv 1$  on  $K$, $\theta=0 $ a.e. outside $2K$ and  $0\leq \theta (x) \leq 1$ for   all $x\in H$. 
\end{lem}
\begin{proof} 
By the $0-1$ law (e.g., \cite[Thm. 2.5.5]{Bo}), the vector space  $E$ spanned by  $K$ has measure  $1$. Consequently,  $\lim_{m\to \infty} \mu(mK) =1$. Fix  $m\in \N$ such that 
$$\mu(mK) >\frac{8}{9}.$$
Let us consider the Minkowski functional defined by   $p_K(x) : = \inf \{ \alpha >0:\;x\in \alpha K\}$ for $x\in E$, and the function
$d(x) : = \inf \{p_K(x-y):\;y\in  K\}$ if  $x\in E$, $d(x)=1$ if  $x\notin E$. We modify it setting 
$$\varphi (x) = 1-  h(d(x)) , \quad x\in H,$$
where $h(t) = t$ for  $ t\leq 1$ and  $h(t)=1$ for  $t\geq 1$. The function  $\varphi $ is Borel measurable, has values between  $0$ and  $1$,  $\varphi \equiv 1$ on  $K$ and  $\varphi \equiv 0$ outside  $E$ and outside $2K $. We regularize it applying  $T_{0}(t) $, 
where  $t>0$ is chosen such that 
$$ 1-e^{-t/2} <\frac{1}{8}, \quad m\sqrt{1-e^{-t}}  <\frac{1}{8}. $$
Since  $\varphi \in {\mathcal B}_b(H)$, then  $T_0(t)\varphi \in W^{\infty}(H, \mu)$ (e.g., \cite[Prop. 5.4.8]{Bo}). 
 
Moreover, 
\begin{equation}
\label{eq:1}
  T_0(t)\varphi(x) \geq \frac{2}{3} \;\forall x\in K, \quad T_0(t)\varphi(x) \leq \frac{3}{5} \;\forall x\in E\setminus 2K.
\end{equation}
Indeed, let $x\in K$. Then $ e^{-t/2}x\in K$, and for each $y\in mK$ we have  $\sqrt{1-e^{-t}}y\in K/8$. The sum  $ e^{-t/2}x+\sqrt{1-e^{-t}}y $ belongs to  $9K/8$, so that  $d(e^{-t/2}x+\sqrt{1-e^{-t}}y)\leq 1/8$ and therefore
$\varphi(e^{-t /2}  x+ \sqrt{1-e^{-t}}y) \geq 7/8$. Since  $\mu(H\setminus mK )\leq 1/9$, we get  $T_0(t)\varphi(x)\geq 7/8 - 1/9 > 2/3$.  
Let now $x\in E\setminus 2K$.  Since  $e^{-t/2} >7/8$, $e^{-t/2}x \notin 7K/4 $ and consequently  for every  $y\in mK$ the sum  $ e^{-t/2}x+\sqrt{1-e^{-t}}y $ does not belong to  $7K/4 - K/8 = 13K/8$. Therefore,  $d( e^{-t/2}x+\sqrt{1-e^{-t}}y) \geq 5/8$, so that $\varphi(e^{-t /2}  x+ \sqrt{1-e^{-t}}y) \leq 3/8$. Again since  $\mu(H\setminus mK )\leq 1/9$, we get   $T_0(t)\varphi(x)\leq 3/8 +1/9 = 35/72 <3/5$, and \eqref{eq:1} is proved. 

Now  fix a function  $\eta \in C^{\infty}(\R)$ such that  $0\leq \eta \leq 1$, $\eta (t) = 0$ for  $t\leq 3/5$, $\eta (t) =1$ for  $t\geq 2/3$, and set 
$$\theta(x) = \eta ( T_0(t)\varphi (x)), \quad x\in H.$$
The function $\theta $ is what we were looking for. It has values between  $0$ and  $1$, it belongs to  $W^{\infty}(H, \mu)$, $\theta(x) =1$ for  $x\in K$  and $\theta(x) =0$ for  $x\in E\setminus 2K$. Since  $\mu(E)=1$, then $\theta(x) =0$ for almost all  $x\in H\setminus 2K$. The statement follows.  
\end{proof}

Now we fix  
 $\varphi_0 \in C^{\infty}_{c}(\R)$ with  $0\leq \varphi_0\leq 1$, $\int_{\R} \varphi_0(t)dt =1$ and  $\varphi_0\equiv 1$ in a neighborhood of  $0$, $\varphi_0\equiv 0$ outside  $(-1, 1)$. Then for each $r\in \R$ the sequence  $\{ \varphi_0(j(t-r))dt/j\}$ converges weakly to the Dirac measure $\delta_r$. 

For each  $r$ in the interior part of $g(H)$ we set
$$\theta_n(x) = \theta\bigg(\frac{x}{n}\bigg) ,\quad x\in H; \quad \quad \varphi_j(t) =  \frac{\varphi_0(j(t-r))}{j}, \quad j\in \N, \;t\in \R.$$
The following proposition is proved in \cite{Bo}. Since in the Hilbert space case  there are not simplifications with respect to the general setting of \cite{Bo}, we refer to \cite[Lemma 6.10.1, Thm. 6.10.2]{Bo} for the proof.  

\begin{prop}
\label{costruzione}
\begin{itemize}
\item[(a)] For each $n\in \N$, the sequence of measures
$$\nu_{n, j}(dx) = \theta_n(x)\frac{\varphi_j(g(x))}{k(g(x))}\mu(dx)$$
converges weakly to a measure $\nu_n$ concentrated on  $\Sigma_r := g^{-1}(r)$. Moreover, for each continuous  $f\in W^{2,2}_{0}(H)$ we have 
\begin{equation}
\label{e0}
\int_H f\, d\nu_n = \int_{ \Sigma_r} f\, d\nu_n = \frac{k_{f\theta_n}(r)}{k(r)}. 
\end{equation}
\item[(b)]  In its turn, the sequence  $\nu_n$ converges weakly to a probability measure  $\sigma^{(g)}_r $ concentrated on $\Sigma_r $, such that for each  continuous  $f\in W^{2,2}_{0}(H)$ we have 
\begin{equation}
\label{e1}
\int_H f\, d\sigma^{(g)}_r  = \int_{ \Sigma_r} f\, d\sigma^{(g)}_r = \frac{k_{f }(r)}{k(r)}. 
\end{equation}
\end{itemize}
\end{prop}

\begin{defn}
\label{MisSup}
For every Borel bounded function  $\varphi :H\mapsto \R$ and for every $r$ in the interior part of $g(H)$ we set 
$$\int_{\Sigma_r} \varphi \,d\sigma_r : = k(r) \int_{\Sigma_r} \varphi|Q^{1/2}Dg| \,d\sigma^{(g)}_r .$$
\end{defn}

\begin{rem}
\label{rem:limite}
It is easy to see that for every    $f:H\mapsto \R$ such that  $f |Q^{1/2}Dg| \in W^{2,2}_{0}(H, \mu)\cap C(H)$ we have
$$\int_{\Sigma_r}f \,d\sigma_r = \lim_{\eps\to 0} \frac{1}{2\eps} \int_{r-\eps \leq g(x)\leq r+\eps} f |Q^{1/2}Dg| \,d\mu .$$
Indeed, applying Proposition \ref{Pr:contk} we get  
$$ \lim_{\eps\to 0} \frac{1}{2\eps} \int_{r-\eps \leq g(x)\leq r+\eps} f |Q^{1/2}Dg| \,d\mu  =  \lim_{\eps\to 0} \frac{1}{2\eps} \int_{r-\eps}^{r+\eps} d(f |Q^{1/2}Dg|\circ \mu)$$
$$ =  \lim_{\eps\to 0} \frac{1}{2\eps} \int_{r-\eps}^{r+\eps}  k_{f |Q^{1/2}Dg|}(t)dt = k_{f |Q^{1/2}Dg|}(r).$$
On the other hand, by Proposition  \ref{costruzione}(b) we have
$$k_{f |Q^{1/2}Dg|}(r) = k(r)\int_{\Sigma_r}  f |Q^{1/2}Dg| \,d\sigma^{(g)}_r $$
and the right hand side is just $\int_{\Sigma_r}f \,d\sigma_r$ by definition. 
 \end{rem}


\end{document}